\documentclass[article,11pt,reqno]{amsart}
\usepackage[top=25mm,bottom=25mm,left=25mm,right=25mm]{geometry}
\usepackage{color}
\usepackage{amsmath}
\usepackage{graphicx}
\usepackage{tikz}

\usetikzlibrary{arrows,decorations.markings}

\usepackage{mathrsfs}
\usepackage{calligra}
\usepackage[mathscr]{eucal}

\newtheorem{lemma}{Lemma}[section]
\newtheorem{proposition}{Proposition}[section]
\newtheorem{theorem}{Theorem}[section]
\newtheorem{corollary}{Corollary}[section]

\newtheorem{remark}{Remark}[section]
\newtheorem{assumption}{Assumption}[section]

\makeatletter
\def\section{\@startsection{section}{1}%
\z@{1\linespacing\@plus\linespacing}{1\linespacing}%
{\bf\centering}}
\def\subsection{\@startsection{subsection}{0}%
\z@{\linespacing\@plus\linespacing}{\linespacing}%
{\bf}}
\makeatother

\makeatletter
\@addtoreset{equation}{section}
\makeatother

\DeclareMathOperator{\Dom}{Dom}
\DeclareMathOperator{\Spec}{Spec}

\DeclareMathOperator{\Per}{Per}

\newcommand{\Norm}[2]{\left\Vert #1 \right\Vert_{#2}}

\providecommand{\pro}[1]{(#1_t)_{t \geq 0}}

\providecommand{\semi}[1]{\{#1_t: t \geq 0\}}

\newcommand{\cD}{\mathcal{D}}

\newcommand{\cK}{\mathcal{K}}
\newcommand{\cV}{\mathscr{V}}

\newcommand{\cB}{\mathcal{B}}

\newcommand{\cJ}{\mathcal{J}}

\newcommand{\cH}{\mathcal{H}}

\newcommand{\cU}{\mathcal{U}}
\newcommand{\R}{\mathbb{R}}
\newcommand{\E}{\mathbb{E}}
\newcommand{\bP}{\mathbb{P}}

\newcommand{\ex}{\mathbb{E}}
\newcommand{\Rd}{\mathbb{R}^d}
\newcommand{\N}{\mathbb{N}}
\definecolor{mr}{rgb}{0.1,0.2,0.7}

\begin{document}
\title[Bulk behaviour of ground states]
{\small Bulk behaviour of ground states for relativistic Schr\"odinger operators
with compactly supported potentials}
\author{Giacomo Ascione and J\'ozsef L{\H o}rinczi}

\address{Giacomo Ascione,
Scuola Superiore Meridionale \\
Universit\`a degli Studi di Napoli Federico II, 80126 Napoli, Italy}
\email{giacomo.ascione@unina.it}
\address{J\'ozsef L\H orinczi,
Alfr\'ed R\'enyi Institute of Mathematics \\
1053 Budapest, Hungary}
\email{lorinczi@renyi.hu}

\begin{abstract}
We propose a probabilistic representation of the ground states of massive and massless Schr\"odinger
operators with a potential well in which the behaviour inside the well is described in terms of the 
moment generating function of the first exit time from the well, and the outside behaviour in terms 
of the Laplace transform of the first entrance time into the well. This allows an analysis of their 
behaviour at short to mid-range from the origin. In a first part we derive precise estimates on these 
two functionals for stable and relativistic stable processes. Next, by combining scaling properties 
and heat kernel estimates, we derive explicit local rates of the ground states of the given family 
of non-local Schr\"odinger operators both inside and outside the well. We also show how this approach 
extends to fully supported decaying potentials. By an analysis close-by to the edge of the potential 
well, we furthermore show that the ground state changes regularity, which depends qualitatively on the 
fractional power of the non-local operator.  


\medskip
\noindent
2010 {\it MS Classification}: Primary 47D08, 60G51; Secondary 47D03, 47G20

\end{abstract}

\maketitle

\baselineskip 0.55 cm

\vspace{-0.5cm}
\tableofcontents

\newpage

\section{Introduction}
The purpose of this paper is to introduce and explore a relationship between the moment generating functions and Laplace transforms
of first hitting times of rotationally symmetric stable and relativistic stable processes, and the ground states of related non-local
Schr\"odinger operators. Making use of this relationship, via precise estimates of these random time functionals we will be able to
derive and prove the spatial localization properties of ground states in the bulk, i.e., for short to middle range from the origin.

The (semi-)relativistic Schr\"odinger operator $H = (-\Delta + m^2)^{1/2}-m + V$ on $L^2(\R^3)$, describing the Hamiltonian of
an electrically charged particle with rest mass $m > 0$ moving under a Coulomb potential $V$ is one of the fundamental models
of mathematical quantum theory, and it has been studied extensively in the literature. Classic papers include \cite{S63,B91,L93}
on the square-root Klein-Gordon equation, \cite{W74,H77,D84,DL} on the properties of the spectrum, stability of the matter
\cite{LY,FL,FLS,LS10}, and eigenfunction decay \cite{CMS}. More recent developments further addressed low-energy scattering
theory \cite{RU16}, embedded eigenvalues and Neumann-Wigner type potentials \cite{LS17}, decay rates when magnetic
potentials and spin are included \cite{HIL13}, a relativistic Kato-inequality \cite{HIL17}, Carleman estimates and unique
continuation \cite{RSV,FF}, or nonlinear relativistic Schr\"odinger equations \cite{ZN,S17,A20}. Given its relationship with
random processes with jumps, the $V=0$ case has received much attention also in potential theory \cite{R02,GR,BMR09}.

There are only a very few examples around for which the spectrum and eigenfunctions of relativistic Schr\"odinger operators are
explicitly determined  \cite{LM,DL}, when the potential is confining rather than decaying, and interesting approximations of
spectra and eigenfunctions for some other cases have been obtained in \cite{KKL}. Thus estimates on the eigenfunctions have a
special relevance. While eigenfunction decay at infinity for a large class of non-local Schr\"odinger operators, including the
relativistic operator, is now understood to a great detail in function of the asymptotic behaviour of the potential
\cite{CMS,HIL13,KL15,KL17}, very little is known on their local behaviour, i.e., for small to medium distances from the origin.
Some information on local properties of eigenfunctions of non-local Schr\"odinger operators with Bernstein functions of the
Laplacian and general potential wells have been obtained in \cite[Sect. 4]{BL19}. Specifically, these include estimates on the
distance of the location of global extrema of eigenfunctions from the edge of the potential well or specific level sets. For
domain operators results in a similar spirit have been obtained in \cite{BKM,B18a}.

Our goal in this paper is to make up for this hiatus and derive the local behaviour of the ground state of the relativistic
operator when $V$ is chosen to be a bounded potential of compact support, and show the extension of our technique to fully
supported potentials. Instead of the above operator, we will consider more generally
\begin{equation*}
H_{m,\alpha} = (-\Delta + m^{2/\alpha})^{\alpha/2}-m + V
\end{equation*}
on $L^2(\R^d)$, with $0 < \alpha < 2$, $m \geq 0$, and $d \in \mathbb N$, and for simplicity we call it in the $m>0$ case the
massive, and for $m=0$ the massless relativistic Schr\"odinger operator. In case $V = -v{\mathbf 1}_\cK$ with a bounded set
$\cK \subset \R^d$ with non-empty interior, we say that $V$ is a potential well with coupling constant (or depth) $v > 0$.

The main idea underlying our approach is simple, and it can be highlighted on the case of a spherical potential well $\cK =
\cB_a$, where $\cB_a$ is a ball of radius $a$ centered in the origin. When the operator $H_{m,\alpha}$ has a ground state
$\varphi_0$ at eigenvalue $\lambda_0 = \inf\Spec H_{m,\alpha} $, a path integral representation gives
\begin{equation*}
e^{-tH_{m,\alpha}}\varphi_0(x) = e^{\lambda_0 t}\ex^x[e^{-\int_0^t V(X_s)ds}\varphi_0(X_t)], \quad t \geq 0,
\end{equation*}
for every point  $x \in \R^d$ (see \cite{LHB}), where now $\int_0^tV(X_s)ds = -v\int_0^t {\mathbf 1}_{\cB_a}(X_s) ds=-vU_t(a)$
is, apart from the constant prefactor, the occupation measure in the ball of the process $\pro X$ starting at $x$, and $\ex^x$
is expectation with respect to its path measure. Clearly, the potential contributes as long as $X_t \in \cB_a$ only, thus we
may consider the first exit time $\tau_{a} = \inf\{t>0: X_t \in \cB_a^c\}$ when starting from the inside, and the first entrance
time $T_{a} = \inf\{t> 0: X_t \in \cB_a\}$ when starting from outside of the well. Since, crucially, $\big(e^{\lambda_0 t}
e^{vU_t(a)}\varphi_0 (X_t)\big)_{t\geq 0}$ can be shown to be a martingale, by optional stopping we get
\begin{equation}
\label{forBM}
\varphi_0(x) =
\; \left\{
\begin{array}{lll}
\ex^x[e^{(v-|\lambda_0|)\tau_{a}}\varphi_0(X_{\tau_{a}})] & \mbox{if $x \in \cB_a$} \vspace{0.2cm} \\
\ex^x[e^{-|\lambda_0|T_{a}}\varphi_0(X_{T_{a}})]  & \mbox{if $x \in \cB_a^c$}.
\end{array}\right.
\end{equation}
When we work with a classical Schr\"odinger operator having $-\frac{1}{2}\Delta$ instead of the relativistic operator,
so that $\pro X = \pro B$ is Brownian motion, due to path continuity the random variables $B_{T_{a}}$ and $B_{\tau_{a}}$ are
supported on the boundary of $\cB_a$, and $\varphi_0$ can be determined exactly. (This is shown in full detail in Section
\ref{primex} below.) When we work with $H_{m,\alpha}$, then $\pro X$ is a jump process and now the supports of $X_{T_{a}}$
and $X_{\tau_{a}}$ spread over the full sets $\cB_a$ and $\cB_a^c$, respectively. Nevertheless, since $|X_{T_{a}}| \leq a$
and $|X_{\tau_{a}}| \geq a$,
using that $\varphi_0$ is (in a spherical potential well, radially) monotone decreasing, the expressions \eqref{forBM}
yield good approximations. Indeed, our main goal in this paper is to derive precise estimates of these functionals and
show how they give tight two-sided bounds on the ground states.
We note that while for the classical Schr\"odinger operator one, though not the only, way to obtain \eqref{forBM} is
through the actual solution of the eigenvalue equation, which is a PDE, this route for $H_{m,\alpha}$ is unworkable as
the solution of a  similar non-local equation is unavailable even for the simplest choices of potential well. Thus the
probabilistic alternative which we develop in this paper will prove to be useful in serving this purpose.

To derive bulk estimates of the ground state, we go through these steps systematically leading to the following main results.
\begin{trivlist}
\item[\, (1)]
\emph{Symmetry properties of the ground state.} It is intuitively clear that the ground state should inherit the
symmetry properties of the potential well, which is also a technically relevant ingredient in deriving local estimates.
In Theorem \ref{thm:rsym} we show rotational symmetry of the ground state when the potential well is a ball, and in
Theorem \ref{thm:asym} reflection symmetry when the potential well has the same symmetry with respect to a hyperplane.


\item[\, (2)]
\emph{Local estimates of the ground state.}
In Theorem \ref{mainth} we prove that, like anticipated above, \eqref{forBM} allow to derive two-sided bounds and the
ground state of $H_{m,\alpha}$ with $m \geq 0$ can be approximated like
\begin{equation}
\label{gsapproxi}
\varphi_0(x) \asymp
\; \left\{
\begin{array}{lll}
\varphi_0(\mathbf{a})\E^x\big[e^{(v-|\lambda_0|)\tau_a}\big] & \mbox{if $x \in \cB_a$} \vspace{0.2cm} \\
\varphi_0(\mathbf{a})\E^x[e^{-|\lambda_0|T_{a}}]  & \mbox{if $x \in \cB_a^c$},
\end{array}\right.
\end{equation}
where $\mathbf{a}=(a,0,\ldots,0)$, and the dependence of the comparability constants on the parameters of the
non-local operator, potential well and spatial dimension can be tracked throughout. By deriving precise two-sided
estimates on the moment generating function of $\tau_a$ and the Laplace transform of $T_{a}$ in Section 3, we
can make the expressions more explicit and obtain
\begin{equation}
\label{expliest}
\frac{\varphi_0(x)}{\varphi_0(\mathbf{a})} \asymp
\; \left\{
\begin{array}{lll}
1+\frac{v-|\lambda_0|}{\lambda_a-v+|\lambda_0|}\left(\frac{a-|x|}{a}\right)^{\alpha/2} 
& \mbox{if \; $x \in \cB_a$} \vspace{0.2cm} \\
j_{m,\alpha}(|x|) & \mbox{if \; $x \in \cB_a^c$},
\end{array}\right.
\end{equation}
see Corollary \ref{cor:asymp}, where $j_{m,\alpha}$ denotes the jump kernel of the operator $L_{m,\alpha}$ (see
details in Section 2.1), and $\lambda_a = \lambda_a(m,\alpha)$ is its principal Dirichlet eigenvalue for $\cB_a$.
While the comparability constants depend on $m$, inside the potential well the $x$-dependence is the same for
both the massless and massive cases, reflecting the fact that the two processes are locally comparable. Since by
using the $L^2$-normalization condition on the ground state the value $\varphi_0(\mathbf{a})$ can further be estimated
from both sides (Proposition \ref{ata}), the right hand side above actually provides bounds on $\varphi_0$ itself, with 
a new proportionality constant (Corollary \ref{cor:asymp2}).
As an application of the information on the local behaviour, in Propositions \ref{moments1}-\ref{moments2} below
we estimate the ground state expectations $\Lambda_p(\varphi_0) = \left(\int_{\R^d} |x|^p \varphi_0^2(x)dx\right)^{1/p}$,
i.e., the moments of the position in the weighted space $L^2(\R^d,\varphi_0^2dx)$ describing the ``halo" or size of the
ground state on different scales.
Finally, in Theorems \ref{thm:main1noncomp}-\ref{thm:mainoutnoncomp} we obtain counterparts of
\eqref{gsapproxi}-\eqref{expliest} to bounded decaying potentials supported everywhere in $\R^d$, giving estimates of
$\varphi_0$ on appropriate level sets of the potential.

Using all this information, we also get some insight into the mechanisms driving these two regimes of behaviour:
\begin{enumerate}
\item[(i)]
\emph{Inside the potential well.}
Since we show that $(a-|x|)^{\frac{\alpha}{2}} \asymp \ex^x[\tau_{a}]$, from
\eqref{expliest} we see that the behaviour of $\varphi_0(x)/\varphi_0(\mathbf{a})$ is essentially determined
by the ratio $\ex^x[\tau_a]/\ex^0[\tau_a]$ of mean exit times. Note that this is different from the case of the
classical Schr\"odinger operator with the same potential well (see Section \ref{primex} below). For Brownian
motion in $\R^d$ it is well known that $\ex^x[\tau_a] = \frac{1}{d}(a^2-|x|^2)$ and the moment generating
function of $\tau_a$ for $d=1$ is given by $\ex^x[e^{u\tau_{a}}] = \cos (\sqrt{2u}x)/\cos(\sqrt{2u} a)$
(and Bessel functions for higher dimensions, see Remark \ref{CT} below), thus the relation
$\varphi_0(x)/\varphi_0(\mathbf{a}) \approx \ex^x[\tau_a]/\ex^0[\tau_a]$ no longer holds and the higher order
moments of $\tau_a$ contribute significantly. The reason for this can be appreciated to be that the
$\alpha$-stable and relativistically $\alpha$-stable processes related to $L_{m,\alpha}$ and $L_{0,\alpha}$,
respectively, have a different nature from Brownian motion. Indeed, we have shown previously that these two
processes satisfy the jump-paring property, i.e., that all multiple large jumps are stochastically dominated
by single large jumps, while Brownian motion evolves through typically small increments and builds up
``backlog events" inflating sojourn times (for the definitions and discussion see \cite[Sect. 2.1]{KL15},
\cite[Def. 2.1, Rem. 4.4]{KL17}). Furthermore, it is also seen from \eqref{expliest} that the ratio between
the maximum $\varphi_0(0)$ of the ground state and $\varphi_0(\mathbf{a})$ is determined by $\frac{\lambda_a}
{\lambda_a-(v-|\lambda_0|)}$, i.e., in fact the ratio of the gap between the ground state energy from the minimum
value of the potential and the energy necessary to climb and leave the well.

\item[(ii)]
\emph{Outside the potential well.}
The behaviour outside is governed by the L\'evy measure which was shown in \cite{KL17} for large enough $|x|$
and we see here by a different approach that this already sets in from the boundary of the potential well. This
is heuristically to be expected due to free motion everywhere outside the well, while to see a ``second order"
contribution of non-locality (distinguishing between polynomially vs exponentially decaying jump measures)
around the boundary of the well would need more refined tools.

\item[(iii)]
\emph{At the boundary of the potential well.} From the profile functions given by \eqref{expliest} it can be
conjectured that, although the ground state is continuous (see Section 2.2 below), its change of behaviour around
the potential well is rather abrupt. Indeed, in Theorem \ref{irreg} and Remark \ref{bdryirreg} we show that at the
boundary $\varphi_0 \not \in C^{\alpha+\delta}_{\rm loc}(\cB_{a+\varepsilon} \setminus \overline{\cB}_{a-\varepsilon})$
for every $\delta \in (0,1-\alpha)$ whenever $\alpha \in (0,1)$, and $\varphi_0 \not \in C^{1,\alpha+\delta-1}_{\rm loc}
(\cB_{a+\varepsilon}\setminus \overline{\cB}_{a-\varepsilon})$ for every $\delta \in (0,2-\alpha)$ whenever $\alpha \in
[1,2)$, for any small $\varepsilon > 0$. This implies that for the range of small $\alpha$ the ground state cannot be
$C^1$ at the boundary, and for values of $\alpha$ starting from 1 it cannot be $C^2$ at the boundary.
\end{enumerate}

\vspace{0.1cm}
\item[\, (3)]
\emph{Entrance/exit time estimates.} All these results depend on precise two-sided estimates on the moment generating
function for exit times from balls, and the Laplace transform of hitting times for balls, which we provide here (Section
3). Clearly, these are of independent interest in probabilistic potential theory; for further applications see \cite{OOT}
on crossing times of subordinate Bessel processes.
\end{trivlist}

For the remaining part of the paper, we proceed in Section 2 to a precise description of the operators and processes,
and in Section 3 to presenting the details of hitting/exit time estimates. Then in Section 4 we show the martingale
property mentioned above and symmetry of the ground state, and in Section 5 derive the local estimates, regularity results
and study the moments of the position in the ground states.

\section{Preliminaries}
\subsection{The massive and massless relativistic operators}
Let $\alpha\in (0, 2)$, $m\geq 0$, $\Phi_{m,\alpha}(z)=(z+m^{2/\alpha})^{\alpha/2}-m$ for every $z \ge 0$, and denote
\begin{eqnarray*}
L_{m,\alpha} \!\!\! &=& \!\!\! \Phi_{m,\alpha}(-\Delta) =(-\Delta+m^{2/\alpha})^{\alpha/2}-m  \quad \mbox{if $m>0$}\\
L_{0,\alpha} \!\!\! &=& \!\!\! \Phi_{0,\alpha}(-\Delta) = (-\Delta)^{\alpha/2}  \hspace{2.6cm} \mbox{if $m=0$}.
\end{eqnarray*}
We will combine the notation into just $L_{m,\alpha}$, $m \geq 0$, when a statement refers to both cases. These operators
can be defined in several possible ways. We define them via the Fourier multipliers
\begin{equation*}
\widehat{(L_{m,\alpha} f)}(y) = \Phi_{m,\alpha}(|y|^2)\widehat f(y), \quad y \in \R^d, \; f \in
\Dom(L_{m,\alpha}),
\end{equation*}
with domain
\begin{equation*}
\Dom(L_{m,\alpha})=\Big\{f \in L^2(\R^d): \Phi_{m,\alpha}(|\cdot|^2) \widehat f \in L^2(\R^d) \Big\},
\quad m\geq 0.
\end{equation*}
Then for $f \in C^\infty_{\rm c}(\R^d)$ the expressions
\begin{equation}\label{intrep}
L_{m,\alpha}f(x) = -\lim_{\varepsilon\downarrow 0} \int_{|y-x|>\varepsilon} \left(f(y)-f(x)\right)\nu_{m,\alpha}(dy)
\end{equation}
hold, with the L\'evy measures
\begin{equation*}
\nu_{m,\alpha}(dx) = j_{m,\alpha}(|x|)dx = \frac{2^{\frac{\alpha-d}{2}} m^{\frac{d+\alpha}{2\alpha}} \alpha}
{\pi^{d/2}\Gamma(1-\frac{\alpha}{2})}
\frac{K_{(d+\alpha)/2} (m^{1/\alpha}|x|)}{|x|^{(d+\alpha)/2}} \, dx, \quad x \in \R^d\setminus \{0\},
\end{equation*}
for $m>0$ (\emph{relativistic fractional Laplacian}), and
\begin{equation*}
\nu_{0,\alpha}(dx) = j_{0,\alpha}(|x|)dx = \frac{2^\alpha \Gamma(\frac{d+\alpha}{2})}{\pi^{d/2}|\Gamma(-\frac{\alpha}{2})|}
\frac{dx}{|x|^{d+\alpha}}, \quad x \in \R^d\setminus \{0\}
\end{equation*}
for $m=0$ (\emph{fractional Laplacian}). Here
\begin{equation*}
\label{bessel3}
K_\rho (z) = \frac{1}{2} \left(\frac{z}{2}\right)^\rho \int_0^\infty t^{-\rho - 1} e^{-t-\frac{z^2}{4t}} dt, \quad z > 0,
\; \rho > -\frac{1}{2}.
\end{equation*}
is the standard modified Bessel function of the third kind. The operator $L_{m,\alpha}$ is positive, and self-adjoint with
core $C^\infty_{\rm c}(\R^d)$, for every $0< \alpha < 2$ and $m \geq 0$.


The difference of the massive and massless operators is bounded, and the relationship can be made explicit, which
will be useful below. For $m, r >0$ denote
\begin{align*}
\label{sigma}
\begin{split}
\sigma_{m,\alpha}(r)
&=
\frac{\alpha 2^{1-\frac{d-\alpha}{2}}}{\Gamma\left(1-\frac{\alpha}{2}\right)
\pi^{\frac{d}{2}}} \left(\frac{2^{\frac{d+\alpha}{2}-1}\Gamma\left(\frac{d+\alpha}{2}\right)}{r^{d+\alpha}}-
\frac{m^{\frac{d+\alpha}{2\alpha}}K_{\frac{d+\alpha}{2}}\left(m^{1/\alpha}r\right)}{r^{\frac{d+\alpha}{2}}}\right)
\\&=
\frac{\alpha 2^{1-\frac{d-\alpha}{2}}}{\Gamma\left(1-\frac{\alpha}{2}\right)\pi^{\frac{d}{2}}}
{\frac{1}{r^{d+\alpha}}}
\int_0^{m^{1/\alpha} r}w^{\frac{d+\alpha}{2}}K_{\frac{d+\alpha}{2}-1}(w)dw,
\end{split}
\end{align*}
and define the measure
\begin{equation*}
\Sigma_{m,\alpha}(A)=\int_{A}\sigma_{m,\alpha}(|x|)dx,
\end{equation*}
for all Borel sets $A\subset \R^d$. It can be shown that $\Sigma_{m,\alpha}$ is finite, positive and has full mass
$\Sigma_{m,\alpha}(\R^d)=m$. For every function $f \in L^\infty(\R^d)$ consider the operator
\begin{equation*}
G_{m,\alpha}f(x)=\frac{1}{2}\int_{\R^d}(f(x+h)-2f(x)+f(x-h))\sigma_{m,\alpha}(|h|)dh.
\end{equation*}
which is well-defined and
$\Norm{G_{m,\alpha}f}{\infty}\le 2m\Norm{f}{\infty}$
holds. Then the decomposition
\begin{equation}\label{dec}
j_{0,\alpha}(r)=j_{m,\alpha}(r)+\sigma_{m,\alpha}(r)
\end{equation}
holds, which implies the formula
\begin{equation*}
\label{withG}
L_{m,\alpha}f=L_{0,\alpha}f-G_{m,\alpha}f,
\end{equation*}
for every function $f$ belonging to the domain of $L_{m,\alpha}$.
For the details and proofs we refer to \cite[Sect. 2.3.2]{AL}, see also \cite[Lem. 2]{R02}.

Next consider the multiplication operator $V: \R^d \to \R$ on $L^2(\R^d)$, which plays the role of the potential.
In case $V = -v\mathbf{1}_\cK$ with a bounded set $\cK \subset \Rd$ having a non-empty interior, we say that $V$
is a potential well with coupling constant $v > 0$. Since such a potential is relatively bounded with respect to
$L_{m,\alpha}$, the operator
\begin{equation}
\label{relsch}
H_{m,\alpha} = L_{m,\alpha} -v{\mathbf 1}_\cK
\end{equation}
can be defined by standard perturbation theory as a self-adjoint operator with core $C_{\rm c}^\infty(\R^d)$.
For simplicity, we call $H_{m,\alpha}$ the (massive or massless) \emph{relativistic Schr\"odinger operator} with
potential well supported in $\cK$, no matter the value of $\alpha \in (0,2)$.

Below we will use the following notations. For two functions $f,g:\R^d \to \R$ we write $f(x)\asymp g(x)$ if there
exists a constant $C \ge 1$ such that $(1/C)g(x)\le f(x)\le Cg(x)$. We denote $f(x)\sim g(x)$ as $|x| \to \infty$
(resp. if $|x| \downarrow 0$) if $\lim_{|x| \to \infty}\frac{f(x)}{g(x)}=1$ (resp. if $\lim_{|x| \downarrow 0}
\frac{f(x)}{g(x)}=1$). Finally, we denote $f(x)\approx g(x)$ as $|x| \to \infty$ (analogously for $|x| \downarrow 0$)
if there exists a constant $C \ge 1$ such that
$(1/C)\le \liminf_{|x| \to \infty} f(x)/g(x)\le \limsup_{|x| \to \infty} f(x)/g(x) \le C$.
Also, we will use the notation $\cB_r(x)$ for a ball of radius $r$ centered in $x\in \R^d$, write just $\cB_r$
when $x=0$, and $\omega_d = |\cB_1|$ for the volume of a $d$-dimensional unit ball.
Moreover,
for a domain $\cD \subset \R^d$ we write $\cD^c$ to denote $\R^d \setminus \overline{\cD}$. In proofs we number
the constants in order to be able to track them, but the counters will be reset in a subsequent statement and proof.
Also, in the statements to follow, we will use the default assumptions $0<\alpha<2$ and $m\geq 0$ implicitly, unless
specified otherwise.

\subsection{Feynman-Kac representation and the related random processes}
The operators $-L_{m,\alpha}$ are Markov generators and give rise to the following L\'evy processes, which can be
realised on the space of c\`adl\`ag paths (i.e., the space of functions that are continuous from the right with
left limits), indexed by the positive semi-axis. To ease the notation, we denote these processes by $\pro X$
without subscripts, and it will be clear from the context which process it refers to. Also, we denote by $\mathbb P^x$
the probability measure on the space of c\`adl\`ag paths, induced by
the process $\pro X$ starting from $x\in \R^d$, by $\ex^x$ expectation with respect to
$\mathbb P^x$, and simplify the notations to $\mathbb P$ and $\ex$ when $x=0$. We will also use the notation
$\ex^x[f(X_t) ; \,\mbox{\small conditions}]$ to mean $\ex^x[f(X_t) \textbf{1}_{\{\mbox{\tiny conditions}\}}]$.

If $m>0$, the operator $-L_{m,\alpha}$ generates a rotationally invariant relativistic $\alpha$-stable process $\pro X$,
and if $m=0$, the operator $-L_{0,\alpha}$ generates a rotationally invariant $\alpha$-stable process $\pro X$. Thus in
either case
\begin{equation*}
P_tf(x):= \left(e^{-tL_{m,\alpha}}f\right)(x) = \ex^x[f(X_t)], \quad x \in \R^d, \, t \geq 0, \, f \in L^2(\R^d),
\end{equation*}
holds, giving rise to the Markov semigroup $\semi P$. Each $P_t$, $t > 0$, is an integral operator with translation invariant
integral kernel $p(t,x,y) := p_t(x-y)$, i.e., $P_tf(x)=\int_{\R^d}p_t(x-y)f(y)dy$ for all $f \in L^p(\R^d)$, $1 \le p \le
\infty$. Also,
\begin{equation*}
\ex[e^{iu\cdot X_t}] = e^{t\Phi_{m,\alpha}(|u|^2)}, \quad u \in \R^d, \, m \geq 0,
\end{equation*}
so that $\Phi_{m,\alpha}(|u|^2)=(|u|^2+m^{2/\alpha})^{\alpha/2}-m$, $m>0$, gives the characteristic exponent of the rotationally
invariant relativistic $\alpha$-stable process, which has the L\'evy jump measure $\nu_{m,\alpha}(dx)$, and $\Phi_{0,\alpha}
(|u|^2)= |u|^\alpha$ gives the characteristic exponent of the rotationally invariant $\alpha$-stable process, which has the
L\'evy jump measure $\nu_{0,\alpha}(dx)$. From a straightforward analysis it can be seen that for small $|x|$ the L\'evy
intensity $j_{m,\alpha}(x)$ behaves like $j_{0,\alpha}(x)$, but due to $K_\rho (x) \sim C|x|^{-1/2}e^{-|x|}$ as $|x| \to
\infty$ for a suitable constant $C>0$, it decays exponentially, while $j_{0,\alpha}(x)$ is polynomial. This difference in
the behaviours has a strong impact on the properties of the two processes.

The main object of interest in this paper are the ground states $\varphi_0$ of the operators $H_{m,\alpha}$ as given by
\eqref{relsch}, i.e., non-zero solutions of the eigenvalue equation
\begin{equation*}
H_{m,\alpha}\varphi_0 = \lambda_0 \varphi_0
\end{equation*}
corresponding to the lowest eigenvalue, so that $\varphi_0 \in \Dom(H_{m,\alpha})\setminus\{0\}$ and $\lambda_0 =
\inf\Spec H_{m,\alpha}$, whenever they exist. Since the potentials $V = -v\textbf{1}_\cK$ are relatively compact perturbations
of $H_{m,\alpha}$, the essential spectrum is preserved, and thus $\Spec H_{m,\alpha} = \Spec_{\rm ess} H_{m,\alpha} \cup
\Spec_{\rm d}H_{m,\alpha}$, with $\Spec_{\rm ess} H_{m,\alpha} = \Spec_{\rm ess} L_{m,\alpha} = [0 ,\infty)$. The existence of
a discrete component depends on further details of the potential. Generally, $\Spec_{\rm d} H_{m,\alpha} \subset (-v, 0)$,
and $\Spec_{\rm d} H_{m,\alpha}$ consists of a finite set of isolated eigenvalues of finite multiplicity each.

For non-positive compactly supported potentials it is known that $\Spec_{\rm d} H_{m,\alpha} \neq \emptyset$ if $\pro X$
is a recurrent process \cite{CMS}, \cite[Th. 4.308]{LHB}, i.e., $H_{m,\alpha}$, $m \geq 0$, does have a ground state
$\varphi_0$ in every such case for all $v > 0$. Recall the Chung-Fuchs criterion of recurrence, which says that for
a process with characteristic exponent $\Psi$ the condition $\int_{|u|<r} \frac{du}{\Psi(u)} < \infty$ for some $r>0$, is
equivalent with the transience of the process  \cite[Cor. 37.17]{Sat}, \cite[Th. 3.84]{LHB}. An application to the processes
above gives that the relativistic $\alpha$-stable process is recurrent whenever $d=1$ or 2, and transient for $d \geq 3$,
while the $\alpha$-stable process is recurrent in case $d=1$ and $\alpha \geq 1$, and transient otherwise. In the transient
cases, \cite[Prop. 2.7]{AL3} guarantees that for sufficiently large $v$ (for instance, $v>\lambda_{\cK}$, where $\lambda_{\cK}$ 
is the principal Dirichlet eigenvalue of $L_{m,\alpha}$ over the well $\cK$) a ground state exists. Furthermore, by 
\cite[Lem. 4.5]{AL3} we know that $v+\lambda_0<\lambda_{\cK}$.

Whenever a ground state $\varphi_0$ of the operator $H_{m,\alpha}$ exists, a Feynman-Kac type representation
\begin{equation}
\label{FKPW}
e^{-tH_{m,\alpha}}\varphi_0(x) = e^{\lambda_0 t}\ex^x[e^{-\int_0^t V(X_s)ds}\varphi_0(X_t)] =
e^{\lambda_0 t}\ex^x[e^{vU_t^\cK(X)}\varphi_0(X_t)],  \quad x \in \Rd, \, t \geq 0
\end{equation}
holds, where
\begin{equation*}
U_t^\cK(X) = \int_0^t {\mathbf 1}_\cK(X_s) ds
\end{equation*}
is the occupation measure of the set $\cK$ by $\pro X$. For the details and proofs we refer to \cite[Sect. 4.6]{LHB}.
For the non-local Schr\"odinger operators $H_{m,\alpha}$ the semigroup $\semi T$, $T_t = e^{-tH_{m,\alpha}}$,
is well-defined and strongly continuous. For all $t>0$, every $T_t$ is a bounded operator on every $L^p(\Rd)$ space,
$1 \leq p \leq \infty$. By \cite[Prop. 4.291]{LHB} the operators $T_t: L^p(\Rd) \to L^p(\Rd)$ for $1 \leq p \leq \infty$,
$T_t: L^p(\Rd) \to L^{\infty}(\Rd)$ for $1 < p \leq \infty$, and $T_t: L^1(\Rd) \to L^{\infty}(\Rd)$ are bounded, for all
$t > 0$. Also, $T_t$ has a bounded measurable integral kernel $q(t, x, y)$ for all $t>0$, i.e., $T_t f(x) = \int_{\Rd}
q(t,x,y)f(y)dy$, for all $f \in L^p(\Rd)$, $1 \leq p \leq \infty$.

Again by \cite[Prop. 4.291]{LHB}, for all $t>0$ and $f \in L^{\infty}(\Rd)$, $T_t f$ is a bounded continuous function.
Thus all the eigenfunctions of $H_{m,\alpha}$ are bounded and continuous, whenever they exist. Also, they have a pointwise
decay to zero at infinity, and the asymptotic behaviour
\begin{equation*}
\label{asympto}
\varphi_0(x) \approx j_{m,\alpha}(x)
\; \left\{
\begin{array}{lll}
= \, {\mathcal A}_{d,\alpha} |x|^{-d-\alpha}          & \mbox{for \, $m = 0$}   \vspace{0.2cm} \\
\approx \, |x|^{-(d+\alpha+1)/2}e^{-m^{1/\alpha} |x|}  & \mbox{for \, $m > 0$}
\end{array}\right.
\end{equation*}
holds, with ${\mathcal A}_{d,\alpha}=\frac{2^\alpha \Gamma(\frac{d+\alpha}{2})}{\pi^{d/2}|\Gamma(-\frac{\alpha}{2})|}$.
For further details we refer to \cite{KL17}. Furthermore, it can be shown that if a ground state exists $\varphi_0$ for
$H_{m,\alpha}$, then due to the positivity improving property of the Feynman-Kac semigroup $\varphi_0$ is unique and
has a strictly positive version, which we will choose throughout this paper. For details we refer to
\cite[Sects. 4.3.2, 4.9.1]{LHB}.

\subsection{Heat kernel of the killed Feynman-Kac semigroup}
Let $\cD \subset \R^d$ be an open
 set and consider the first exit time
\begin{equation}
\label{firstex}
\tau_\cD = \inf\left\{t > 0: X_t \notin \cD\right\}
\end{equation}
from $\cD$. When $\cD = \cB_R$ we simplify the notation to $\tau_R$, while if $\cD=\cB_R^c$ we use $T_R$. (From the
context the reader will realise the meanings and not confuse this simple notation with the semigroup operators $T_t$.)
The transition probability densities $p_\cD(t,x,y)$ of the process killed on exiting $\cD$ (or heat kernel of the killed
semigroup) are given by the Dynkin-Hunt formula
\begin{align}
\label{Hunt}
p_\cD(t,x,y)= p_t(x-y) - \ex^x\left[p_{t-\tau_\cD}(y-X_{\tau_\cD}); \tau_\cD < t\right], \quad x , y \in \cD.
\end{align}
The heat kernel $p_\cD(t,x,y)$ gives rise to the killed Feynman-Kac semigroup $\semi
{P^{\cD}}$ by $P_t^{\cD} f(x) = \int_\cD p_\cD(t,x,y)f(y)dy$, for all $x\in\cD$, $t>0$ and $f\in L^2(\R^d)$. It is known
that $\semi {P^{\cD}}$ is a strongly continuous semigroup of contraction operators on $L^2(\cD)$ and every operator
$P_t^{\cD}$, $t>0$, is self-adjoint.

Below we will make frequent use of the Ikeda-Watanabe formula \cite[Th. 1]{IW}, which says that for every $\eta>0$ and
every bounded or non-negative Borel function $f$ on $\R^d$, the equality
\begin{align*}
\ex^x\left[e^{-\eta \tau_\cD} f(X_{\tau_\cD})\right] = \int_\cD \int_0^{\infty} e^{-\eta t} p_\cD(t,x,y) dt
\int_{\cD^c} f(z) j_{m,\alpha}(z-y) dzdy, \quad x \in \cD,
\end{align*}
holds. The same arguments leading to the above expression also allow the more general formulation (see, for instance,
\cite[eq. (1.58)]{BBKRSV09} and \cite[Th. 2]{IW})
\begin{align}
\label{IW}
\ex^x\left[f(\tau_{\cD}, X_{\tau_\cD-}, X_{\tau_\cD})\right] = \int_\cD
\int_{\cD^c} \int_0^{\infty} p_\cD(t,x,y) f(t,y,z) j_{m,\alpha}(z-y) dtdzdy, \quad x \in \cD,
\end{align}
which holds for every bounded or non-negative Borel function $f:[0,\infty] \times \R^d \times \R^d \to \R$. We will keep
referring to this as the Ikeda-Watanabe formula.

In what follows we will also rely on some estimates of the heat kernel of the killed semigroup. By \eqref{Hunt}, clearly
$p_\cD(t,x,y) \leq p_t(x-y)$ for all $t > 0$ and $x,y \in \cD$. Recall that the semigroup $\semi {P^{\cD}}$ is said to be
intrinsically ultracontractive (IUC) whenever there exists $C_t^\cD > 0$ such that
$p_\cD(t,x,y) \leq C_t^\cD f_\cD(x) f_\cD(y)$, for all $t > 0$ and $x,y \in \cD$, where $f_\cD$ is the principal Dirichlet
eigenfunction of the operator $L_{m,\alpha}$ in the domain $\cD$. It can be shown that if $\semi {P^{\cD}}$ is IUC, then a
similar lower bound holds with another constant.
The following result provides a bound on $p_t(x)$, and will be useful for the IUC property of $\semi {P^{\cD}}$ for a
class of domains $\cD$ that we will use below.
\begin{lemma}
For every $\delta>0$ there exists a constant $C_{d,m,\alpha}(\delta)$ such that
\begin{equation*}
\sup_{|x| \ge \delta, \ t>0}p_t(x)\le C_{d,m,\alpha}(\delta).
\end{equation*}
\end{lemma}
\begin{proof}
Fix $\delta>0$. By \cite[eq. (9)]{R02} we know that
\begin{equation*}
p_t(x)\le C^{(1)}_\alpha e^{mt}t\frac{2^{\alpha}\Gamma\left(\frac{d+\alpha}{2}\right)}{\pi^{d/2} |x|^{d+\alpha}}.
\end{equation*}
Thus for $t \le 1$ and $|x|\ge \delta$ we obtain
\begin{equation*}
p_t(x)\le C^{(1)}_\alpha e^{m}\frac{2^{\alpha}\Gamma\left(\frac{d+\alpha}{2}\right)}
{\pi^{d/2}\delta^{d+\alpha}}=:C^{(2)}_{d,m,\alpha}(\delta).
\end{equation*}
For $t \ge 1$ we distinguish two cases. If $m=0$, we use the estimate (see, for instance, \cite{BBKRSV09})
\begin{equation*}
	p_t(x)\le C^{(3)}_{d,\alpha} t^{-\frac{d}{\alpha}}\le C^{(3)}_{d,\alpha}, \quad t>1.
\end{equation*}
If $m>0$, we can use \cite[Lem. 3]{R02} to conclude that
\begin{equation*}
p_t(x)\le C^{(4)}_{d,m,\alpha} \left(m^{\frac{d}{\alpha}-\frac{d}{2}}{t^{-\frac{d}{2}}}+
t^{-\frac{d}{\alpha}}\right) \le C^{(5)}_{d,m,\alpha}, \quad t>1.
\end{equation*}
Hence we can define
\begin{equation*}
C_{d,m,\alpha}(\delta)=\begin{cases}
\max\big\{C^{(2)}_{d,0,\alpha}(\delta), C^{(3)}_{d,\alpha}\big\} & \; m=0 \vspace{0.2cm} \\
\max\big\{C^{(2)}_{d,m,\alpha}(\delta), C^{(4)}_{d,m,\alpha}\big\} & \; m>0, \\
\end{cases}
\end{equation*}
giving $\sup_{|x|\ge \delta, \ t>0}p_t(x)\le C_{d,m,\alpha}(\delta)$ for every $m \ge 0$.
\end{proof}
Using that $\nu_{m,\alpha}(\cB_r(x))>0$ for every $x \in \R^d$, $r>0$ and $m \ge 0$, we immediately get the
following result from the previous lemma and \cite[Th. 3.1]{G08}.
\begin{corollary}
\label{lem:ultra}
Let $\cD$ be a bounded Lipschitz domain. The killed semigroup $\semi {P^{\cD}}$ is IUC.
\end{corollary}
\noindent

We will denote the principal Dirichlet eigenfunction of $L_{m,\alpha}$ by $f_R$ at eigenvalue $\lambda_R$ whenever
$\cD=\cB_R$. Using IUC and its implication of a similar lower bound, and the continuity of the killed heat kernel,
it can be shown \cite[Th. 4.2.5]{D89} that there exists a large enough $T>0$ such that
\begin{equation}
\label{abovebelow}
\frac{1}{2} \, e^{-\lambda_R t}f_R(x)f_R(y)\le p_{\cB_R}(t,x,y)\le \frac{3}{2}\, e^{-\lambda_R t}f_R(x)f_R(y),
\end{equation}
for all $t > T$ and $x,y \in \cB_R$.

\section{Exit and hitting times estimates}
\subsection{Estimates on the survival probability}
As we will see below, the local behaviour of ground states depends on a function which can be estimated by
using tools of potential theory for the stable and relativistic stable processes. We will denote this by
$\cV_{\alpha,m}$ and call it rate function. In this section we derive some key information on this function
first. The results contained in this subsection have been obtained in a more general context in \cite{BGR2015}.
Since here we are considering two specific cases, which are widely used in applications, we reconsider some of
the proofs in order to identify the values of the involved constants, which are not explicit in the cited work
due to the greater generality of the arguments involved.

\begin{lemma}\label{lem:meanexit}
Let $\cD$ be a $C^{1,1}$ bounded open set in $\R^d$, $\pro {X^{(0)}}$ be an isotropic $\alpha$-stable process and
$\pro {X^{(m)}}$ be an isotropic relativistic $\alpha$-stable process with mass $m>0$. Consider the first exit time
$\tau^{(m)}_{\cD}=\inf\{t>0: X_t^{(m)}\not \in \cD\}, \ m \ge 0$.
Then $\E^x[\tau^{(m)}_{\cD}]\asymp \E^x[\tau^{(0)}_{\cD}]$, for every $m>0$ and the comparability constant is independent
of $\cD$.
\end{lemma}
\begin{proof}
The statement easily follows from \cite[Cor. 1.2]{CKS10} and \cite[Th. 1.3]{CKS12} due to the comparability of the
respective Green functions.
\end{proof}
As a consequence, we get the following upper bound.
\begin{corollary}\label{cor:eigcontr}
We have $\lambda_R R^\alpha \le C_{d,m,\alpha}$.
\end{corollary}
\begin{proof}
Denote $s(x)=\E^x[\tau_R]$ and $S=\Norm{s}{L^2(\cB_R)}$. First consider the case $m=0$. Then the explicit formula
due to M. Riesz (e.g., \cite[eq. (1.56)]{BBKRSV09})
\begin{equation*}
s(x)=\frac{\pi^{1+d}\Gamma\left(\frac{d}{2}\right)
\sin\left(\pi \frac{\alpha}{2}\right)\left|\Gamma\left(-\frac{\alpha}{2}\right)\right|}
{2^{\alpha}\Gamma\left(\frac{d+\alpha}{2}\right)}(R^2-|x|^2)^{\alpha/2}, \quad |x|\le R,
\end{equation*}
holds. Hence we have
\begin{equation}\label{eq:contrmeanexit}
s(x) \ge C_{d,m,\alpha} R^{\alpha}, \quad |x|\le \frac{R}{2}.
\end{equation}
Lemma \ref{lem:meanexit} guarantees that \eqref{eq:contrmeanexit} holds even for $m>0$. Thus in general we have
$S^2 \ge C_{d,m,\alpha}R^{2\alpha}|\cB_{R/2}|$.
By \cite[Prop. 2.1]{BK04} and Schwarz inequality we then obtain
\begin{equation*}
\lambda_R \le \int_{\cB_R}\frac{s(x)}{S^2}dx\le \sqrt{\frac{|B_R|}{S^2}} \le \frac{C_{d,m,\alpha}}{R^{\alpha}}.
\end{equation*}
\end{proof}

We say that a function $f:\R^d \to \R$ is \emph{$(m,\alpha)$-harmonic} on an open set $\cD \subset \R^d$ if for every
open set $\cU \subset\subset \cD$ (i.e., $\overline{\cU} \subset \cD$ is compact) the equality $f(x)=\E^x[f(X_{\tau_{\cU}})]$
holds for every $x \in \cU$. In the following we come back to the notation by $\pro X$ meaning either of the processes for the
massless and massive cases, as used previously.
\begin{lemma}
\label{lem:harm}
Let $d=1$ and fix $r_0>0$. There exist an increasing concave (and thus subadditive) $(m,\alpha)$-harmonic function
$\cV_{m,\alpha}(r): (0,\infty)\to\R^+$ and constants $0<C^{(1)}_{m,\alpha,r_0}<C^{(2)}_{m,\alpha,r_0}$ such that
\begin{equation*}
C^{(1)}_{m,\alpha,r_0}r^{\alpha/2}\le \cV_{m,\alpha}(r)\le C^{(2)}_{m,\alpha,r_0}r^{\alpha/2}, \quad 0 \le r\le r_0.
\end{equation*}
\end{lemma}
\begin{proof}
Consider the running supremum $M_t=\sup_{0 \le s \le t}X_t$, and let $Y_t=M_t-X_t$ be the process obtained by reflecting
$X_t$ on hitting the supremum. Let $A_t$ be the local time at zero of $Y_t$, and $Z_t =\inf\{\tau>0: \ A_\tau>t \}$ its
right-continuous inverse.
Also, consider $H_t=M_{Z_t}$. By \cite[eq. (1.8)]{S80} there exists a function $\psi_{m,\alpha}$ such that
$\int_0^{\infty}\psi_{m,\alpha}(s)f(s)ds=\int_0^{\infty}\E[f(H_s)]ds$,
for every non-negative Borel function $f$. Choosing in particular $f=\mathbf{1}_{[0,r]}$, we define
\begin{equation*}
\cV_{m,\alpha}(r)=\int_0^{r}\psi_{m,\alpha}(\rho)d\rho=\int_0^{\infty}\bP(H_\rho \le r)d\rho.
\end{equation*}
Note that $\pro H$ is a subordinator (see \cite[Lem. VI.2]{B96}), different from a Poisson process since $(0,\infty)$ is a
regular domain for $\pro X$. We can define its inverse subordinator $H^{-1}_t:=\inf\{s > 0: H_s>t\}$ and observe that
$\cV_{m,\alpha}(t)=\E[H^{-1}_t]$,
implying subadditivity of $\cV_{m,\alpha}$ (see \cite[Ch. III]{B96}). The fact that $\cV_{m,\alpha}$ is $(m,\alpha)$-harmonic
in $(0,\infty)$ follows from \cite[Th. 2]{S80}. The comparability result follows by \cite[Prop. 2.2, Ex. $2.3$]{KSV09}.
Concavity results by \cite[Prop. 2.1]{KSV09} and \cite[Th. 10.3]{SSV12} as $\psi_{m,\alpha}=\cV'_{m,\alpha}$ is non-increasing.
\end{proof}
\begin{remark}\label{rmk:asymp}
{\rm
In fact, $\cV_{0,\alpha}(r)=r^{\alpha/2}$. Moreover, for $m>0$ again by \cite[Prop. $2.2$ and Ex. 2.3]{KSV09} we get
$\cV_{m,\alpha}(r)\sim r$ as $r \to \infty$. As a direct consequence of the monotone density theorem, we furthermore have
$\psi_{m,\alpha}(r)\sim r^{\frac{\alpha}{2}-1}$ as $r \downarrow 0$, for all $m\geq 0$.
}
\end{remark}
As a consequence, we obtain the following Harnack-type inequality.
\begin{lemma}
\label{lem:Harnack}
For every $0<x \le y \le z \le 5x$ we have
\begin{equation*}
\cV_{m,\alpha}(z)-\cV_{m,\alpha}(y)\le 5\cV_{m,\alpha}'(x)(z-y).
\end{equation*}
\end{lemma}
\begin{proof}
By Lemma \ref{lem:harm} we know that $\cV_{m,\alpha}$ is concave and thus, in particular, log-concave. Hence the result
follows by \cite[Lem. 7.1]{BGR2015}.
\end{proof}

Moreover, we can use the function $\cV_{m,\alpha}$ to derive the following estimate.
\begin{corollary}
\label{halfspacecoro}
Let $d=1$ and define $\tau_{(0,\infty)}=\inf\{t>0: \ X_t\le 0\}$. There exist constants
$C^{(1)}_{m,\alpha}$ and $C^{(2)}_{m,\alpha}$ such that
\begin{equation*}
C^{(1)}_{m,\alpha}\left(\frac{r^{\alpha/2}}{\sqrt{t}}\wedge 1\right)\le \bP^r(\tau_{(0,\infty)}>t)\le
C^{(2)}_{m,\alpha}\left(\frac{r^{\alpha/2}}{\sqrt{t}}\wedge 1\right)
\end{equation*}
\end{corollary}
\begin{proof}
Immediate by \cite[Cor. 3.2]{KMR13} and Lemma \ref{lem:harm}.
\end{proof}
\begin{remark}
\rm{  
In the case $m=0$ it is not difficult to determine explicitly the constant given in Corollary \ref{cor:eigcontr}, 
while it is clear that the upper and lower bounds in Lemma \ref{lem:harm} are actually identities. Furthermore, 
the constants obtained in Corollary \ref{halfspacecoro} can be computed exactly to be $C^{(1)}_{m,\alpha}=\frac{1}{2e}
\left(\frac{e-1}{8e^2}\right)^2$ and $C^{(2)}_{m,\alpha}=\frac{e}{e-1}$, which are independent of $m$ and $\alpha$. 
In fact, as observed in \cite{BGR2015}, these constants are universal for more general unimodal symmetric L\'evy 
processes. The constants given in the following statements can be, at least in the case $m=0$, tracked from the 
cited results or numerically evaluated via the principal Dirichlet eigenfunction.
}
\end{remark}
As a direct consequence of Lemmas \ref{lem:harm}-\ref{lem:Harnack}, we obtain the following lower bound.
\begin{proposition}\label{prop:arg}
For every $R>0$ there exist constants $C^{(1)}_{d,m,\alpha,R}, C^{(2)}_{d,m,\alpha}$ such that
\begin{equation*}
\bP^x(\tau_R >t)\ge C^{(1)}_{d,m,\alpha,R}\left(\frac{(R-|x|)^{\alpha/2}}{\sqrt{t}}\wedge 1\right),
\quad t \le C^{(2)}_{d}\cV_{m,\alpha}^2(R).
\end{equation*}
\end{proposition}
\begin{proof}
By Lemma \ref{lem:Harnack} and \cite[Prop. 6.1]{BGR2015} we know that there exist constants $C^{(2)}_{d},
C^{(3)}_d > 0$ such that
\begin{equation*}
\bP^x(\tau_R >t)\ge C^{(3)}_{d,R}\left(\frac{\cV_{m,\alpha}(R-|x|)}{\sqrt{t}}\wedge 1\right),
\quad t \le C^{(2)}_{d}\cV_{m,\alpha}^2(R).
\end{equation*}
Lemma \ref{lem:harm} then completes the proof.
\end{proof}
Furthermore, we can derive an upper bound on the survival probability $\tau_R$.
\begin{lemma}\label{lem:upex}
	For every $x \in \cB_R$ and $t>0$ we have
	\begin{equation*}
	\bP^x(\tau_R> t)\le 2 \Big(\frac{(R-|x|)^{\alpha/2}}{\sqrt{t}}\wedge 1\Big).
	\end{equation*}
\end{lemma}
\begin{proof}
	Since $\pro{X}$ is rotationally symmetric, we may choose $x=r\mathbf{e}_1$ without loss of generality, where
	$\mathbf{e}_1=(1,0,\dots,0)$ and $r \in (0,R)$. Let $\cH_R^{\leftarrow}:=\{x \in \R^d: \ x_1< R\}$ and consider
	the first exit time $\widetilde{\tau}_R:=\inf\{t>0: \ X_r \in (\cH_R^{\leftarrow})^c\}$. Since $\cB_R \subseteq
	\cH_R$, we have $\tau_R \le \widetilde{\tau}_R$ almost surely. With the same notation $\tau_{(0,\infty)}$ as in
	Corollary \ref{halfspacecoro}, it follows that
	\begin{equation*}
	\bP^x(\tau_R>t)\le \bP^x(\widetilde{\tau}_R>t)=\bP^{(r-R)\mathbf{e}_1}(\widetilde{\tau}_0>t)=
	\bP^{r-R}(\tau_{(0,\infty)}>t)\le 2\left(\frac{(R-r)^{\alpha/2}}{\sqrt{t}}\wedge 1\right).
	\end{equation*}
\end{proof}
Using intrinsic ultracontractivity of the killed semigroup, we can improve these estimates.
\begin{proposition}
	\label{prop:contrsurvexit}
	For every $x \in \cB_R$, we have
	\begin{equation*}
	\bP^x(\tau_R> t)\asymp e^{-\lambda_R t}\left(\frac{(R-|x|)^{\alpha/2}}{\sqrt{t} \wedge
		R^{\alpha/2}}\wedge 1\right),
	\end{equation*}
	where the comparability constants depend on $d,m,\alpha,R$, and $\lambda_R$ is the principal Dirichlet eigenvalue of
	$L_{m,\alpha}$ in the ball $\cB_R$.
\end{proposition}
\begin{proof}
Since we have already recalled Lemma \ref{lem:upex} and Proposition \ref{prop:arg}, we only need to prove the exponential
domination for large values of $t>0$. Let $f_R$ be the principal Dirichlet eigenfunction of $L_{m,\alpha}$ for the ball
$\cB_R$ and observe that, by \cite[Prop. 4.289]{LHB}, $f_R$ is continuous and bounded. Since the killed semigroup is IUC,
see Lemma \ref{lem:ultra}, we can choose $T>0$ such that \eqref{abovebelow} holds for every $t \ge 0$ and $x,u \in \cB_R$.
For this fixed $T$, by \cite[Th. 1.1]{CKS10} and \cite[Th. 1.1]{CKS12}, it follows that there exists a constant
$C^{(1)}_{d,m,\alpha,R}>0$ such that for every $t \ge T$ and $x,u \in \cB_R$
\begin{align}\label{eq:sharpheatcontr}
\frac{1}{C^{(1)}_{d,m,\alpha,R}}e^{-\lambda_R t}(R-|x|)^{\frac{\alpha}{2}}(R-|u|)^{\frac{\alpha}{2}}\le p_{\cB_R}(t,x,u)
\le \frac{3}{2}e^{-\lambda_R t}(R-|x|)^{\frac{\alpha}{2}}(R-|u|)^{\frac{\alpha}{2}}
\end{align}
holds. Combining \eqref{abovebelow} and \eqref{eq:sharpheatcontr} we have, for all $x,u \in \cB_R$,
\begin{equation*}
\label{eq:ineqDir}
f_R(x)f_R(u)\ge \frac{2}{3C^{(1)}_{d,m,\alpha,R}}(R-|x|)^{\frac{\alpha}{2}}(R-|u|)^{\frac{\alpha}{2}}.
\end{equation*}
Taking $x=u=0$, the previous inequality gives
\begin{equation}\label{eq:ineqDir1}
f_R(0)^2\ge \frac{2}{3C^{(1)}_{d,m,\alpha,R}}R^\alpha>0.
\end{equation}
Furthermore, choosing $u=0$ in \eqref{eq:ineqDir1} we get
\begin{equation}\label{eq:lowerDir}
f_R(x)\ge \frac{2}{3C^{(1)}_{d,m,\alpha,R}f_R(0)}R^\frac{\alpha}{2}(R-|x|)^{\frac{\alpha}{2}}=:
C^{(2)}_{d,m,\alpha,R}(R-|x|)^{\frac{\alpha}{2}}.
\end{equation}
Finally, by \eqref{abovebelow} and \eqref{eq:lowerDir} we obtain the lower bound
\begin{align*}
1 \ge \bP^x(\tau_R>t)
&=
\int_t^{\infty}\int_{\cB_R^c}\int_{\cB_R}j_{m,\alpha}(|z-u|)p_{\cB_R}(s,x,u)dudzds\\
&\ge
\frac{e^{-\lambda_R t}}{2\lambda_R}\int_{\cB_R^c}\int_{\cB_R}j_{m,\alpha}(|z-u|) f_R(x)f_R(u)dudz\\
&\ge
\frac{C^{(2)}_{d,m,\alpha,R}(R-|x|)^{\frac{\alpha}{2}}e^{-\lambda_R t}}{2\lambda_R}
\int_{\cB_R^c}\int_{\cB_R}j_{m,\alpha}(|z-u|)f_R(u)dudz.
\end{align*}
This guarantees that
\begin{equation*}
\int_{\cB_R^c}\int_{\cB_R}j_{m,\alpha}(|z-u|)f_R(u)dudz<\infty
\end{equation*}
and, at the same time,
\begin{align*}
\bP^x(\tau_R>t)\ge
\frac{C^{(2)}_{d,m,\alpha,R}(R-|x|)^{\frac{\alpha}{2}}e^{-\lambda_R t}}{2\lambda_R}\int_{\cB_R^c}
\int_{\cB_R}j_{m,\alpha}(|z-u|)f_R(u)dudz=:C^{(3)}_{d,m,\alpha,R}(R-|x|)^{\frac{\alpha}{2}}e^{-\lambda_R t},
\end{align*}
for every $x \in \cB_R$ and $t \ge T$. Similarly, we have the estimate from above,
\begin{align*}
\bP^x(\tau_R>t)
&\le  \frac{3}{2}\int_t^{\infty}e^{-\lambda_R s}ds\int_{\cB_R^c}\int_{\cB_R}j_{m,\alpha}(|z-u|)f_R(x)f_R(u)dudz\\
&\le \frac{3 \Norm{f_R}{\infty}}{2\lambda_R}e^{-\lambda_R t}\int_{\cB_R^c}\int_{\cB_R}j_{m,\alpha}(|z-u|)f_R(u)dudz=:
C^{(4)}_{d,m,\alpha,R}e^{-\lambda R t}.
\end{align*}
\end{proof}
Next we derive an upper bound for the function $\bP^x(T_R>t)$. First we need a technical lemma.
\begin{lemma}\label{lem:outball}
There exists a constant $C_{d,m,\alpha}>0$ such that
\begin{equation*}
	\nu_{m,\alpha}(\cB^c_r) \sim C_{d,m,\alpha}r^{-\alpha}, \quad r \downarrow 0.
\end{equation*}
\end{lemma}
\begin{proof}
There is nothing to prove if $m=0$, thus take $m>0$ and for all $\varepsilon>0$ let $t_0(\varepsilon)$ such that
$(1-\varepsilon)C^{(1)}_{d,m,\alpha}\rho^{-d-\alpha}\le j_{m,\alpha}(\rho)\le (1+\varepsilon)C^{(1)}_{d,m,\alpha}
\rho^{-d-\alpha}$ for every $0<\rho<t_0(\varepsilon)$ (note that this holds by the $0+$ asymptotics of the Bessel
function). Consider $r<t_0 (\varepsilon)$ and observe that
\begin{equation*}
\nu_{m,\alpha}(\cB_r^c)
= \int_r^{t_0(\varepsilon)}\rho^{d-1}j_{m,\alpha}(\rho)d\rho+\int_{t_0(\varepsilon)}^{\infty}\rho^{d-1}
j_{m,\alpha}(\rho)d\rho=:I_1(\varepsilon,r)+I_2(\varepsilon).
\end{equation*}
Clearly, $I_2(\varepsilon)<\infty$. Since
\begin{equation*}
(1-\varepsilon)\frac{C^{(1)}_{d,m,\alpha}}{\alpha}r^{-\alpha}-(1-\varepsilon)
\frac{C^{(3)}_{d,m,\alpha}}{\alpha}t_0(\varepsilon)^{-\alpha}\le I_1(\varepsilon,r)
\le
(1+\varepsilon)\frac{C^{(1)}_{d,m,\alpha}}{\alpha}r^{-\alpha}-(1+\varepsilon)
\frac{C^{(3)}_{d,m,\alpha}}{\alpha}t_0(\varepsilon)^{-\alpha},
\end{equation*}
the result follows directly.
\end{proof}
\begin{proposition}
\label{prop:upboundT}
For every $0<R<R_0$ there exists a constant $C_{d,m,\alpha,R,R_0}>0$ such that
\begin{equation}
\label{eq:propunbound}
\bP^x(T_R>t)\le C_{d,m,\alpha,R,R_0}\frac{(|x|-R)^{\alpha/2}}{\sqrt{t}\wedge R^{\alpha/2}},
\quad |x| \in [R,R_0).
\end{equation}
\end{proposition}
\begin{proof}
Consider the function
\begin{equation*}
\cJ_{m,\alpha}(R)=\inf_{0 \le r \le R}\nu_{m,\alpha}(\cB_r^c)\cV_{m,\alpha}^2(r).
\end{equation*}
Observe that $\nu_{m,\alpha}(\cB_r^c)\cV_{m,\alpha}^2(r)>0$ for every $r>0$. Moreover, by Lemmas
\ref{lem:harm} and \ref{lem:outball} we know that $\nu_{m,\alpha}(\cB_r^c)\cV_{m,\alpha}^2(r) \ge C_{m,\alpha,r_0}>0$
for $r_0>0$ and $r \in (0,r_0)$. This implies $\cJ_{m,\alpha}(R)>0$. Lemma \ref{lem:Harnack} guarantees that
\cite[Lem. 6.2]{BGR2015} applies and we obtain
\begin{equation*}
\bP^x(T_R>t)\le \frac{5C_{d}}{(\cJ(R))^2}\frac{\cV_{m,\alpha}(|x|-R)}{\sqrt{t}\wedge \cV_{m,\alpha}(R)}.
\end{equation*}
Finally, for $|x| \in (R,R_0)$ we can use Lemma \ref{lem:harm} to complete the proof.
\end{proof}
\begin{remark}
{\rm
Note that in case $m=0$, there exists a constant $C_{d,\alpha}>0$ such that $\cJ_{0,\alpha}(R)\ge C_{d,\alpha}$ for
every $R$. This follows from the asymptotic behaviour of $\nu_{0,\alpha}(\cB_r^c)$ as $r \to \infty$ given in
\cite[Cor. 2.1]{AL}. Thus for the massless case \eqref{eq:propunbound} holds for all $|x| \ge R$, with no dependence
on $R_0$. On the other hand, for $m>0$ we have $\lim_{R \to \infty}\cJ_{m,\alpha}(R)=0$. This is due to $\cV_{m,\alpha}
(R) \sim R$ as $R \to \infty$, as seen in Remark \ref{rmk:asymp}, while $\overline{\nu}_{m,\alpha}(\cB_R^c)$ decays
exponentially (see \cite[Cor. 2.2]{AL}).
}
\end{remark}

\subsection{Estimates on the moment generating function for the exit time from a ball}
In view of deriving and using expressions of the type \eqref{forBM} in our main analysis below, in this section first
we derive estimates of exponentials of exit times of the L\'evy processes $\pro{X}$ for balls and their complements.
Recall \eqref{firstex} and denote by
\begin{equation}
\label{pdfoftau}
g_{\tau_R}(t)=\int_{\cB_R^c}\int_{\cB_R}j_{m,\alpha}(|z-u|)p_{\cB_R}(t,x,u)dudz, \quad t>0,
\end{equation}
the probability density of $\tau_R$.
Now we prove the following estimate for the moment generating function of $\tau_R$.
\begin{theorem}\label{thm:CompMomGenFun}
Fix $R>0$. Then for every $0 \le \lambda<\lambda_R$ and $x \in \cB_R$ we have
\begin{equation*}
\E^x[e^{\lambda \tau_R}-1]\asymp \frac{\lambda}{\lambda_R-\lambda}\left(\frac{R-|x|}{R}\right)^{\alpha/2},
\end{equation*}
where the comparability constant depends on $d,m,\alpha,R$. Moreover, $\E^x[e^{\lambda \tau_R}]=\infty$ whenever
$\lambda \ge \lambda_R$.
\end{theorem}
\begin{proof}
First fix $0 \le \lambda < \lambda_R$. Using \eqref{pdfoftau} and integrating by parts we obtain
\begin{equation}
\label{eq:momexitpass1}
\E^x[e^{\lambda \tau_R}-1]=\int_0^{\infty}(e^{\lambda t}-1)g_{\tau_R}(t)dt=
-\lim_{s \to \infty}(e^{\lambda s}-1)\bP^x(\tau_R>s)+\lambda\int_0^{\infty}e^{\lambda t}\bP^x(\tau_R>t)dt.
\end{equation}
Note that the limit is zero since by Proposition \ref{prop:contrsurvexit}
\begin{equation*}
e^{\lambda s}\bP^x(\tau_R>s)\le C^{(1)}_{d,m,\alpha,R} \, e^{(\lambda-\lambda_R)s}
\frac{(R-|x|)^{\alpha/2}}{\sqrt{s}\wedge R^{\alpha/2}},
\end{equation*}
and $\lambda<\lambda_R$.

First we show the lower bound of the remaining integral at the right hand side of \eqref{eq:momexitpass1}.
Using Proposition \ref{prop:contrsurvexit} again, we get
\begin{align}\label{eq:momexitpass2}
\begin{split}
\int_0^{\infty}e^{\lambda t}\bP^x(\tau_R>t)dt
&\ge
C^{(2)}_{d,m,\alpha,R}\left(\frac{R-|x|}{R}\right)^{\alpha/2}\int_{R^\alpha}^{\infty}e^{-(\lambda_R-\lambda)t}dt\\
&\ge
C^{(2)}_{d,m,\alpha,R}\left(\frac{R-|x|}{R}\right)^{\alpha/2}\frac{1}{\lambda_R-\lambda}e^{-\lambda_RR^\alpha}.
\end{split}
\end{align}
Next note that by Corollary \ref{cor:eigcontr} we have $\lambda_R R^\alpha\le  C^{(3)}_{d,m,\alpha}$ with a
constant $C^{(3)}_{d,m,\alpha}$, thus $e^{-\lambda_R R^\alpha}\ge C^{(4)}_{d,m,\alpha}$. Using this lower bound
in \eqref{eq:momexitpass2} we get
\begin{equation*}
\int_0^{\infty}e^{\lambda t}\bP^x(\tau_R>t)dt\ge C^{(5)}_{d,m,\alpha,R}\frac{1}{\lambda_R-\lambda}
\left(\frac{R-|x|}{R}\right)^{\alpha/2}.
\end{equation*}

To get the upper bound, we estimate
\begin{align*}
\begin{split}
\int_0^{\infty}
& e^{\lambda t}\bP^x(\tau_R>t)dt\le C^{(1)}_{d,m,\alpha,R}\int_0^{\infty}e^{-(\lambda_R-\lambda)t}
\Big(1 \wedge \frac{(R-|x|)^{\alpha/2}}{\sqrt{t}\wedge R^{\alpha/2}}\Big)dt\\
&=
C^{(1)}_{d,m,\alpha,R}\left(\int_0^{R^{\alpha/2}}\Big(1 \wedge \frac{(R-|x|)^{\alpha/2}}{\sqrt{t}}\Big)
e^{-(\lambda_R-\lambda)t}dt
+ \Big(\frac{R-|x|}{R}\Big)^{\alpha/2}\frac{e^{-(\lambda_R-\lambda)R^\alpha}}{\lambda_R-\lambda}\right)\\
&\le
C^{(1)}_{d,m,\alpha,R}\left((R-|x|)^{\alpha/2}\int_0^{R^\alpha}\frac{e^{-(\lambda_R-\lambda)t}}{\sqrt{t}}dt+
\Big(\frac{R-|x|}{R}\Big)^{\alpha/2}\frac{1}{\lambda_R-\lambda}\,\right)\\
&=
C^{(1)}_{d,m,\alpha,R}\left((R-|x|)^{\alpha/2}\left(2R^{\alpha/2}e^{-(\lambda_R-\lambda)R^\alpha}
+2(\lambda_R-\lambda)\int_0^{R^\alpha}\sqrt{t}e^{-(\lambda_R-\lambda)t}dt\right)\right.\\
&\left.\qquad+
\left(\frac{R-|x|}{R}\right)^{\alpha/2}\frac{1}{\lambda_R-\lambda}\right)\\
&\le
C^{(1)}_{d,m,\alpha,R}\left(4(R-|x|)^{\alpha/2}R^{\alpha/2}+\left(\frac{R-|x|}{R}\right)^{\alpha/2}
\frac{1}{\lambda_R-\lambda}\right)\\
&\le
C^{(1)}_{d,m,\alpha,R}\left(4R^\alpha \left(\frac{R-|x|}{R}\right)^{\alpha/2}
\frac{\lambda_R}{\lambda_R-\lambda}+\left(\frac{R-|x|}{R}\right)^{\alpha/2}\frac{1}{\lambda_R-\lambda}\right)
\le
\frac{C^{(5)}_{d,m,\alpha,R}}{\lambda_R-\lambda}\left(\frac{R-|x|}{R}\right)^{\alpha/2},
\end{split}
\end{align*}
where we used the bound $\lambda_R R^\alpha \le  C^{(3)}_{d,m,\alpha}$ again in the last line. This proves the
first part of the claim.

To obtain the second statement we only need to prove that $\E[e^{\lambda_R \tau_R}]=\infty$. Notice that by
Proposition \ref{prop:contrsurvexit}
\begin{equation*}
e^{\lambda_R s}\bP^x(\tau_R>s)\le C^{(1)}_{d,m,\alpha,R} \, \frac{(R-|x|)^{\alpha/2}}{\sqrt{s}\wedge R^{\alpha/2}}.
\end{equation*}
For $s>R^\alpha$ we get
\begin{align*}
\E^x[e^{\lambda_R \tau_R}] &\ge \E^x[e^{\lambda_R \tau_R}-1; \tau_R \le s] = \int_0^s(e^{\lambda t}-1)g_{\tau_R}(t)dt\\
&=
-(e^{\lambda_R s}-1)\bP^x(\tau_R>s)+\lambda_R\int_0^s e^{\lambda_R t}\bP^x(\tau_R>t)dt\\
&\ge
-C^{(1)}_{d,m,\alpha,R}\frac{(R-|x|)^{\alpha/2}}{R^{\alpha/2}}+\lambda_R\int_{R^\alpha}^s e^{\lambda_R t}
\bP^x(\tau_R>t)dt.
\end{align*}
Taking the supremum over $s$ on the right-hand side and using the lower bound in Proposition \ref{prop:contrsurvexit},
we obtain
\begin{align*}
\E^x[e^{\lambda_R \tau_R}] &\ge -C^{(1)}_{d,m,\alpha,R}\frac{(R-|x|)^{\alpha/2}}{R^{\alpha/2}}
+\lambda_R\int_{R^\alpha}^{\infty} e^{\lambda_R t}\bP^x(\tau_R>t)dt\\
&\ge
-C^{(1)}_{d,m,\alpha,R}\frac{(R-|x|)^{\alpha/2}}{R^{\alpha/2}}+C^{(2)}_{d,m,\alpha,R}\lambda_R
\int_{R^\alpha}^{\infty}\frac{(R-|x|)^{\alpha/2}}{R^{\alpha/2}}dt=\infty.
\end{align*}
\end{proof}

\subsection{Estimates on the Laplace transform of the hitting time for a ball}

Next we consider $T_R=\inf\{t>0: \ X_t \in \cB_R\}$ and derive estimates on the Laplace transform
$\E^x[e^{-\lambda T_R}]$, in which case there is no handy tool such as intrinsic ultracontractivity of the
killed semigroup. We start with a lower bound for points in domains of the type $R\le |x| \le R'$, for the
remaining choices of domains see Remark \ref{remcontr1} (2) below.
\begin{theorem}\label{thm:lowboundLapundercrit}
Let $\lambda, \, R>0$ and $R_2>R_1>R$. There exists a constant
$C_{d,m,\alpha,R_1,R_2,R,\lambda}>0$ such that
\begin{equation*}
\E^x[e^{-\lambda T_R}]\ge C_{d,m,\alpha,R_1,R_2,R,\lambda} \; j_{m,\alpha}(|x|), \quad R_1 \le |x|\le R_2.
\end{equation*}
\end{theorem}
\begin{proof}
Define
\begin{equation*}
C^{(1)}_{d,m,\alpha,R_1,R_2}
=\min_{R_1 \leq |x| \leq R_2}\frac{j_{m,\alpha}\left(|x|+\frac{5}{2}R\right)}{j_{m,\alpha}(|x|)}.
\end{equation*}
As before, fix $x=r \mathbf{e}_1$ for $r>0$, and define
$A(x)=\big\{ u \in \R^d: \ |x|+R <|u|<|x|+2R, \ \langle u, \mathbf{e}_1\rangle < 0 \big\}$.
Since $R_1 \le |x| \le R_2$, taking $\cD=\cB_{3R_2}\setminus {\bar\cB_{R}}$ we see that $x \in \cD \subset
{\bar\cB_{R}}^c$. In particular, $p_{\cB_R^c}(t,x,u)\ge p_{\cD}(t,x,u)$. Since $\cD$ is a bounded and open
Lipschitz set, the semigroup with kernel $p_{\cD}(t,x,u)$ is IUC and we can apply to it the lower bound
\eqref{abovebelow} with some $T>0$, and the principal Dirichlet eigenvalue and eigenfunction $\lambda_\cD$
and $f_\cD$ of $L_{m,\alpha}$ on $\cD$. Then by using the Ikeda-Watanabe formula we get
\begin{align}\label{eq:low}
\begin{split}
\E^x[e^{-\lambda T_R}]
&\ge
\E^x\Big[e^{-\lambda T_R}; X_{T_R-}\in A(x), \ |X_{T_R}|<\frac{R}{2}, \ T_R>T\Big]\\
&=
\int_T^{\infty}\int_{\cB_{R/2}}\int_{A(x)}e^{-\lambda t}j(|u-z|)p_{\cB_R^c}(t,x,u)dtdzdu\\
&\ge
\frac{R^d \omega_d}{2^{d+1}} j\Big(|x|+\frac{5}{2}R\Big)\int_T^{\infty}\int_{A(x)}e^{-(\lambda+\lambda_\cD) t}
f_\cD(x)f_\cD(u)dtdu\\
&\ge
\frac{C^{(1)}_{d,m,\alpha,R_1,R_2}R^d \omega_de^{-(\lambda+\lambda_\cD)T}}{(\lambda+\lambda_\cD)2^{d+1}}
j\left(|x|\right)f_\cD(x) \int_{A(x)}f_\cD(u)du.
\end{split}
\end{align}
Note that since $\cD$ is a bounded $C^{1,1}$ domain, by \cite[Th. 1.1]{CKS10} and \cite[Th. 1.1]{CKS12}
there exists a constant $C^{(2)}_{d,m,\alpha,R_1,R_2}>1$ such that for every $t \ge 1$
\begin{equation*}
p_\cD(t,x,u)\le C^{(2)}_{d,m,\alpha,R_1,R_2}e^{-\lambda_\cD}\delta_\cD^{\alpha/2}(x)\delta_\cD^{\alpha/2}(u),
\end{equation*}
holds, where $\delta_\cD(x)={\rm dist}(x,\partial \cD)$. By definition of $f_\cD(x)$ we get
\begin{align}\label{eq:esteigen}
\begin{split}
f_\cD(x)
&=e^{\lambda_\cD}\int_{\cD}p_\cD(1,x,u)f_\cD(u)du\\
&\le C^{(2)}_{d,m,\alpha,R_1,R_2} \Norm{f_\cD}{L^\infty(\cD)}\delta_\cD^{\alpha/2}(x)
\int_{\cD}\delta_\cD^{\alpha/2}(u)du\\
&\le C^{(2)}_{d,m,\alpha,R_1,R_2} \Norm{f_\cD}{L^\infty(\cD)}(6R_2)^{\alpha/2}((3R_2)^d-R_1^d)
\omega_d\delta_\cD^{\alpha/2}(x).
\end{split}
\end{align}
To obtain a lower bound on $f_\cD(x)$, consider $\tau_\cD=\inf\{t>0: \ X_t \in {\bar \cD}^c\}$ and use again
\eqref{IW}, \eqref{eq:esteigen} and the fact that $\delta_\cD(u)\le |u-z|$ for all $z \in \cD^c$, giving
\begin{align*}
\bP^x(\tau_\cD >T)
&=
\int_T^{\infty}\int_{\cD^c}\int_{\cD}j_{m,\alpha}(|z-u|)p(t,x,u)dudzdt\\
&\le
\frac{3}{2}f_\cD(x)\int_T^{\infty}\int_{\cD^c}\int_{\cD}j_{m,\alpha}(|z-u|)e^{-\lambda_\cD t}f_\cD(u)dudzdt\\
&\le \frac{3}{2}f_\cD(x)C^{(2)}_{d,m,\alpha,R_1,R_2} \Norm{f_\cD}{L^\infty(\cD)}(6R_2)^{\alpha/2}
((3R_2)^d-R_1^d)\omega_d\\
&\qquad \times \int_T^{\infty}\int_{\cD^c}\int_{\cD}j_{m,\alpha}(|z-u|)e^{-\lambda_\cD t}\delta_\cD(u)dudzdt\\
&\le C^{(4)}_{d,\alpha}\frac{3e^{-\lambda_\cD T}}{2\lambda_\cD}f_\cD(x)C^{(2)}_{d,m,\alpha,R_1,R_2}
\Norm{f_\cD}{L^\infty(\cD)}(6R_2)^{\alpha/2}((3R_2)^d-R_1^d)\omega_d\\
&\qquad \times \int_{\cD^c}\int_{\cD}\frac{dudz}{|z-u|^{d+\frac{\alpha}{2}}}\\
&\le
\frac{3\Per_{\alpha}(\cD)}{2}C^{(3)}_{d,\alpha} C^{(2)}_{d,m,\alpha,R_1,R_2} \Norm{f_\cD}{L^\infty(\cD)}(6R_2)^{\alpha/2}
((3R_2)^d-R_1^d)\omega_d\frac{e^{-\lambda_\cD T}}{\lambda_D}f_\cD(x),
\end{align*}
where
$\Per_\alpha(\cD)=\int_{\cD}\int_{\cD^c}\frac{dzdu}{|z-u|^{d+\frac{\alpha}{2}}}$
is the fractional perimeter of $\cD$ (see, e.g., \cite{FMM11}), and we used that $j_{m,\alpha}(|z-u|)\le
j_{0,\alpha}(|z-u|)=C^{(3)}_{d,\alpha}|z-u|^{-d-\alpha}$ by \eqref{dec}, see \cite[Lem. 2]{R02}. Hence
$f_\cD(x)\ge C^{(4)}_{d,m,\alpha,R_1,R_2}\bP^x(\tau_\cD >T)$,
where
\begin{equation*}
C^{(4)}_{d,m,\alpha,R_1,R_2}=\frac{2\lambda_\cD e^{\lambda_\cD T}}{3C^{(3)}_{d,\alpha}
\Per_\alpha(\cD)C^{(2)}_{d,m,\alpha,R_1,R_2} \Norm{f_\cD}{L^\infty(\cD)}(6R_2)^{\alpha/2}((3R_2)^d-R_1^d)\omega_d}.
\end{equation*}
Note that $\cD$ is a $C^{1,1}$ bounded set with scaling radius $R_3=(3R_2+R)/2$. Fix $x \in \cD$. Then there exists a
point $\bar{x} \in \cD$ and a ball $\cB_{R_3}(\bar{x})$ such that $x \in \cB_{R_3}(\bar{x})$ and $\delta_\cD(x)=
R_3-|x-\bar{x}|$. By Proposition \ref{prop:contrsurvexit} and the fact that $\cB_{R_3}(\bar{x})\subset \cD$, we know
that there exists a constant $C^{(5)}_{d,m,\alpha,R_1,R_2}$ such that
\begin{equation*}
	\bP^x(\tau_\cD>T)
	\ge \bP^x(\tau_{\cB_{R_3}(\bar{x})}>T)=\bP^{x-\bar{x}}(\tau_{R_3}>T) \ge
	C^{(5)}_{d,m,\alpha,R_1,R_2} e^{-\lambda_\cD T}
	\Big(\frac{\delta_\cD(x)^{\alpha/2}}{\sqrt{T}\wedge R_3^{\alpha/2}}\wedge 1\Big),
\end{equation*}
and then
\begin{equation*}
f_\cD(x)\ge
C^{(6)}_{d,m,\alpha,R_1,R_2}\Big(\frac{\delta_\cD(x)^{\alpha/2}}{\sqrt{T}\wedge R_3^{\alpha/2}} \wedge 1\Big),
\end{equation*}
where $C^{(6)}_{d,m,\alpha,R_1,R_2}=C^{(5)}_{d,m,\alpha,R_1,R_2}C^{(5)}_{d,m,\alpha,R_1,R_2} e^{-\lambda_\cD T}$.
Applying this to \eqref{eq:low} we have
\begin{align*}
\begin{split}
\E^x[e^{-\lambda T_R}]&\ge C^{(7)}_{d,m,\alpha,R_1,R_2,R,\lambda}
\Big(\frac{\delta_\cD(x)^{{\alpha/2}}}{\sqrt{T}\wedge R_3^{\alpha/2}}\wedge 1\Big)j_{m,\alpha}\left(|x|\right)\int_{A(x)}
\Big(\frac{\delta_\cD(u)^{{\alpha/2}}}{\sqrt{T}\wedge R_3^{\alpha/2}}\wedge 1\Big)du,
\end{split}
\end{align*}
where
\begin{equation*}
C^{(7)}_{d,m,\alpha,R_1,R_2,R,\lambda}=\frac{C^{(1)}_{d,m,\alpha}R^d\omega_d}{(\lambda+\lambda_\cD)2^{d+1}}
e^{-(\lambda+\lambda_\cD)T}(C^{(6)}_{d,m,\alpha,R_1,R_2,R})^2.
\end{equation*}
Recall that $\min_{R_1\le |x|\le R_2}\delta_\cD(x)=(C^{(8)}_{R_1,R_2,R})^{\frac{2}{\alpha}}>0$ by definition of $\cD$.
Moreover, $u \in A(x)$ implies $R<R_1+R\le |u|\le R_2+2R<3R_2$, and hence $\min_{u \in A(x)}\delta_\cD(u) \ge
\min_{R_1+R\le |u|\le 2R+R_2}\delta_\cD(u)=(C^{(9)}_{R_1,R_2,R})^{\alpha/2}>0$. Finally, recall also that
$|A(x)|\ge \frac{\omega_d}{2}d(R_1+R)^{d-1}R$
to conclude that
\begin{equation*}
\E^x[e^{-\lambda T_R}]\ge C_{d,m,\alpha,R_1,R_2,R,\lambda} \, j_{m,\alpha}(|x|),
\end{equation*}
where
\begin{equation*}
C_{d,m,\alpha,R_1,R_2,R,\lambda}=C^{(7)}_{d,m,\alpha,R_1,R_2,R,\lambda}\Big(\frac{
C^{(8)}_{R_1,R_2,R}}{\sqrt{T}\wedge R_3^{\alpha/2}}\wedge 1\Big)\Big(\frac{
C^{(9)}_{R_1,R_2,R}}{\sqrt{T}\wedge R_3^{\alpha/2}}\wedge 1\Big)\frac{dR\omega_d}{2}(R_1+R)^{d-1}.
\end{equation*}
\end{proof}
To extend the lower bound up to the boundary of $\cB_R$, we need the following result.
\begin{proposition}\label{prop:halflower}
The following properties hold:
\begin{enumerate}
\item
There exist $R^{(0)}_{d,m,\alpha,R,\lambda}>R$ and $C_{d,m,\alpha,R,\lambda}>0$ such that for every $R \le |x| \le
R^{(0)}_{d,m,\alpha,R,\lambda}$
\begin{equation*}
\E^x[1-e^{-\lambda T_R}]\le C_{d,m,\alpha,R,\lambda}(|x|-R)^{\alpha/2}.
\end{equation*}
\item
There exists $\widetilde{R}_{d,m,\alpha,R,\lambda}>R$ such that for every $R \le |x|
\le \widetilde{R}_{d,m,\alpha,R,\lambda}$
\begin{equation*}
\E^x[e^{-\lambda T_R}]\ge \frac{1}{2}.
\end{equation*}
\end{enumerate}
\end{proposition}
\begin{proof}
By Proposition \ref{prop:upboundT}
\begin{equation}
\label{eq:infprob}
\bP^x(T_R=\infty)\le C^{(1)}_{d,m,\alpha,R}\Big(1 \wedge \frac{(|x|-R)^{\alpha/2}}{R^{\alpha/2}}\Big),
\end{equation}
hence there exists $R^{(0)}_{d,m,\alpha,R,\lambda}>R$ such that, for $R<|x|<R^{(0)}_{d,m,\alpha,R,\lambda}$,
\begin{equation*}
\bP^x(T_R=\infty)\le C^{(1)}_{d,m,\alpha,R} \Big(\frac{R^{(0)}_{d,m,\alpha,R,\lambda}}{R}-1\Big)^{\alpha/2}
<\frac{1}{3},
\end{equation*}
so that
\begin{equation*}
\bP^x(T_R<\infty)\ge 1-C^{(1)}_{d,m,\alpha,R} \Big(\frac{R^{(0)}_{d,m,\alpha,R,\lambda}}{R}-1\Big)^{\alpha/2}
:=C^{(2)}_{d,m,\alpha,R}>\frac{2}{3}.
\end{equation*}
Notice that
\begin{equation*}
\label{eq:expon}
\E^x[e^{-\lambda T_R}]=\E^x[e^{-\lambda T_R}; T_R<\infty]=\E^x[e^{-\lambda T_R}|T_R<\infty]\bP^x(T_R<\infty).
\end{equation*}
Denote $\widetilde{\bP}^x(\cdot)=\bP^x(\, \cdot \,| \, T_R<\infty)$. We have
\begin{align*}
1-\widetilde{\E}^x[e^{-\lambda T_R}]
&=
\int_0^{\infty}\widetilde{\bP}^x(1-e^{-\lambda T_R}>s)ds = \int_0^{1}\widetilde{\bP}^x(1-e^{-\lambda T_R}>s)ds.
\end{align*}
Writing $s=1-e^{-\lambda t}$ we obtain
\begin{align*}
1-\widetilde{\E}^x[e^{-\lambda T_R}]
&=
\int_0^{\infty}\lambda e^{-\lambda t}\widetilde{\bP}^x(1-e^{-\lambda T_R}>1-e^{-\lambda t})dt =
\int_0^{\infty}\lambda e^{-\lambda t}\widetilde{\bP}^x(T_R>t)dt\\
&=
\frac{\lambda}{\bP^x(T_R<\infty)}\int_0^{\infty} e^{-\lambda t}\bP^x(T_R>t, T_R<\infty)dt\\
&\le
\frac{\lambda}{C^{(2)}_{d,m,\alpha,R}}\int_0^{\infty} e^{-\lambda t}\bP^x(T_R>t)dt.
\end{align*}
Using Proposition \ref{prop:upboundT} gives
\begin{equation*}
\bP^x(T_R>t)\le C^{(1)}_{d,m,\alpha,R}\left(1\wedge \frac{(|x|-R)^{{\alpha/2}}}{\sqrt{t}
\wedge R^{{\alpha/2}}}\right), \quad t>0.
\end{equation*}
so that, setting $C^{(3)}_{d,m,\alpha,R}=C^{(1)}_{d,m,\alpha,R}/C^{(2)}_{d,m,\alpha,R}$, we get
\begin{eqnarray*}
\label{eq:Laptran}
\lefteqn{
1-\widetilde{\E}^x[e^{-\lambda T_R}]} \nonumber \\
&\le&
\lambda C^{(3)}_{d,m,\alpha,R}\int_0^{\infty}e^{-\lambda t}
\left(1 \wedge \frac{(|x|-R)^{\alpha/2}}{\sqrt{t}\wedge R^{\alpha/2}}\right)dt \nonumber\\
&=&
\lambda C^{(3)}_{d,m,\alpha,R}\left(\int_0^{(|x|-R)^\alpha}
e^{-\lambda t}dt + (|x|-R)^{\alpha/2}
\Big(\int_{(|x|-R)^\alpha}^{R^\alpha}\frac{e^{-\lambda t}}{\sqrt{t}}dt +
\int_{R^\alpha}^{\infty}\frac{e^{-\lambda t}}{ R^{\alpha/2}}dt\Big) \right) \nonumber \\
&=&
\lambda C^{(3)}_{d,m,\alpha,R}\left(\frac{1-e^{-\lambda (|x|-R)^\alpha}}{\lambda}
+\frac{(|x|-R)^{\alpha/2}}{ \lambda R^{\alpha/2}}e^{-\lambda R^\alpha}
+\int_{(|x|-R)^\alpha}^{R^\alpha}e^{-\lambda t} \frac{(|x|-R)^{\alpha/2}}{\sqrt{t}}dt\right).
\end{eqnarray*}
The last term above can be further estimated as
\begin{eqnarray*}
(|x|-R)^{\alpha/2} \int_{(|x|-R)^\alpha}^{R^\alpha} \frac{e^{-\lambda t}dt}{\sqrt{t}}
&=&
2(|x|-R)^{\alpha/2}(e^{-\lambda R^\alpha}R^{\alpha/2}-e^{-\lambda (|x|-R)^\alpha}
(|x|-R)^{\alpha/2})\\
&& \quad
+2(|x|-R)^{\alpha/2}\int_{(|x|-R)^\alpha}^{R^\alpha}\lambda e^{-\lambda t}\sqrt{t}dt\\
&\le &
2(|x|-R)^{\alpha/2}(e^{-\lambda R^\alpha}R^{\alpha/2}-e^{-\lambda (|x|-R)^\alpha}
(|x|-R)^{\alpha/2})\\
&& \quad
+2(|x|-R)^{\alpha/2}R^{\alpha/2}(e^{-\lambda (|x|-R)^\alpha}-e^{-\lambda R^\alpha})\\
&=&
2(|x|-R)^{\alpha/2}e^{-\lambda (|x|-R)^\alpha}(R^{\alpha/2}-(|x|-R)^{\alpha/2}).
\end{eqnarray*}
In sum, we obtain
\begin{align*}
1-\widetilde{\E}^x[e^{-\lambda T_R}]
&\le
\lambda C^{(3)}_{d,m,\alpha,R}\left((|x|-R)^\alpha+\frac{(|x|-R)^{\alpha/2}}
{ \lambda R^{\alpha/2}}e^{-\lambda R^\alpha}
+2(|x|-R)^{\alpha/2}e^{-\lambda (|x|-R)^\alpha}R^{\alpha/2}\right)\\
&:=C^{(4)}_{d,m,\alpha,R,\lambda}(|x|-R)^\alpha.
\end{align*}
We can complete the proof of part (1) by observing that
\begin{align*}
\E^x[1-e^{-\lambda T_R}]&=\widetilde{\E}^x[1-e^{-\lambda T_R}]\bP^x(T_R < \infty) + \bP^x(T_R=\infty)\\
&\le C_{d,m,\alpha,R,\lambda}(|x|-R)^\alpha,\qquad  R \le |x| \le R^{(0)}_{d,m,\alpha,R,\lambda},
\end{align*}
where we made use of \eqref{eq:infprob}. Part (2) follows from (1) by choosing $R<\widetilde{R}_{d,m,\alpha,R,\lambda}
< R^{(0)}_{d,m,\alpha,R,\lambda}$ so that $\E^x[1-e^{-\lambda T_R}]\le 1/2$ holds for all $R \le |x| \le
\widetilde{R}_{d,m,\alpha,R,\lambda}$.
\end{proof}

Finally, we can combine Theorem \ref{thm:lowboundLapundercrit} with Proposition \ref{prop:halflower} to obtain the
following.
\begin{corollary}
\label{cor:goodlowbound}
Let $R_2>R$. Then there exists a constant $C_{d,m,\alpha,R_2,R,\lambda}$ such that
\begin{equation*}
\E^x[e^{-\lambda T_R}]\ge C_{d,m,\alpha,R_2,R,\lambda}j_{m,\alpha}(|x|), \quad R \le |x| < R_2.
\end{equation*}
\end{corollary}
\begin{proof}
Let $\widetilde{R}_{d,m,\alpha,R,\lambda}$ be defined as in Proposition \ref{prop:halflower}. Then we have
\begin{equation*}
\E^{x}[e^{-\lambda T_R}]\ge \frac{1}{2}
=
\frac{j_{m,\alpha}(|x|)}{2j_{m,\alpha}(|x|)}\ge \frac{j_{m,\alpha}(|x|)}{2j_{m,\alpha}(R)},
\quad R \le |x|<\widetilde{R}_{d,m,\alpha,R,\lambda}.
\end{equation*}		
Combining this estimate with Theorem \ref{thm:lowboundLapundercrit} for $R_1=\widetilde{R}_{d,m,\alpha,R,\lambda}$
the result follows.
\end{proof}

To obtain an upper bound for the same quantities we can make use of \cite[Th. 3.3]{KL17}, particularized to
the massless and massive relativistic stable processes.
\begin{theorem}
\label{thm:upboundLap}
Let $\lambda, R>0$. There exists a constant $C_{d,m,\alpha,R,\lambda}>0$ such that
\begin{equation*}
\E^x[e^{-\lambda T_R}]\le C_{d,m,\alpha,R,\lambda} \; j_{m,\alpha}(|x|), \quad |x|\ge R.
\end{equation*}
\end{theorem}
\begin{proof}
By \cite[Th. 3.3]{KL17} it follows that there exist constants $R^{(1)}_{d,\alpha,m,\lambda,R}>R$ and
$C^{(1)}_{d,\alpha,m,\lambda,R}>0$ such that
\begin{equation*}
\E^x[e^{-\lambda T_R}]\le C^{(1)}_{d,\alpha,m,\lambda,R} \; j_{m,\alpha}(|x|),
\quad |x|\ge R^{(1)}_{d,\alpha,m,\lambda,R}.
\end{equation*}
Let $R^{(2)}_{d,\alpha,m,\lambda,R}=R^{(1)}_{d,\alpha,m,\lambda,R}+1$ and notice that $j_{m,\alpha}(|x|)\ge
j_{m,\alpha}(R^{(2)}_{d,\alpha,m,\lambda,R})$ whenever $R \le |x| \le R^{(2)}_{d,\alpha,m,\lambda,R}$. Hence
for every $R \le |x| \le R^{(2)}_{d,\alpha,m,\lambda,R}$ we get
\begin{equation*}
\E^x[e^{-\lambda T_R}]\le 1 \le \frac{j_{m,\alpha}(|x|)}{j_{m,\alpha}(R^{(2)}_{d,\alpha,m,\lambda,R})}.
\end{equation*}
Setting $C_{d,m,\alpha,R,\lambda}=\max\left\{C^{(1)}_{d,\alpha,m,\lambda,R},
\frac{1}{j_{m,\alpha}(R^{(2)}_{d,\alpha,m,\lambda,R})}\right\}$ completes the proof.
\end{proof}

\begin{remark}
\label{remcontr1}
\hspace{100cm}
{\rm
\begin{enumerate}
\item
A similar upper estimate follows by using the Ikeda-Watanabe formula. In this approach we can derive a bound which is
uniform with respect to $\alpha \in [\alpha_0,2]$ for a suitable $\alpha_0>0$.
\item
Above we obtained a global upper and a local lower bound for $\E^x[e^{-\lambda T_R}]$. A global lower bound for
$\E^x[e^{-\lambda T_R}]$ outside the well will be obtained as a consequence of the estimates of the ground states.
\end{enumerate}
}
\end{remark}

\section{Basic qualitative properties of ground states}
\subsection{Martingale representation of ground states}
For our purposes below it will be useful to consider a variant of the Feynman-Kac representation \eqref{FKPW} with
general stopping times. In order to obtain this, the following martingale property will be important. Define the
random process $\pro {M^x}$,
\begin{equation}
M_t^x=e^{\lambda_0t} e^{-\int_0^t V(X_r+x) dr}\varphi_0(X_t+x), \quad x \in \R^d.
\label{gsmartingale}
\end{equation}
Note that by the eigenvalue equation $\ex[M_t^x]=\varphi_0(x)$, for all $t\geq 0$ and $x\in\R^d$. Let
$\pro {{\mathcal F}^X}$ be the natural filtration of the L\'evy process $\pro X$. A version of the following result
dates back at least to Carmona's work (see \cite[Sect. 4.6.3]{LHB} for a detailed discussion and references), but
since it is of fundamental interest in this paper, we provide a proof for a self-contained presentation.
\begin{lemma}
\label{mar1}
$(M_t^x)_{t\geq 0}$ is a martingale\index{martingale} with respect to $\pro {{\mathcal F}^X}$.
\end{lemma}
\begin{proof}
We have
\begin{equation*}
\E[|M_t^x|] = \E[M_t^x] \leq e^{\lambda_0 t} \|\varphi_0\|_\infty \ex \left[ e^{-\int_0^t V(X_r+x) dr} \right]
\leq e^{{(v-|\lambda_0|) t}} \|\varphi_0\|_\infty  < \infty, \quad t \geq 0.
\end{equation*}
Let $0 \leq s \leq t$. By the Markov property of $\pro X$ we have that
\begin{eqnarray*}
\ex[M_t^x|{\mathcal F}^X_s]
&=&
e^{\lambda_0 t}e^{-\int_0^s V(X_r+x) dr}´\ex [e^{-\int_s^t V(X_r+x) dr}\varphi_0(X_t+x)|{\mathcal F}^X_s] \\
&=&
e^{\lambda_0 s }e^{-\int_0^s V(X_r+x) dr} \ex^{X_s} [e^{\lambda_0(t-s)}e^{-\int_0^{t-s} V(X_r+x) dr}
\varphi_0(X_{t-s}+x)]\\
&=&
e^{\lambda_0 s}e^{-\int_0^s V(X_r+x)dr} \varphi_0(X_s+x) = M_s^x.
\end{eqnarray*}
Hence the lemma follows.
\end{proof}
This martingale property easily leads to the following Feynman-Kac type formula for the stopped process.
\begin{proposition}\label{prop:stop}
Let $\tau$ be a $\mathbb P$-almost surely finite stopping time with respect to the filtration $\pro {{\mathcal F}^X}$.
Then
\begin{equation*}
\varphi_0(x)=\E^x\big[e^{-\int_0^{\tau}(V(X_s)-\lambda_0)ds}\varphi_0(X_\tau)\big].
\end{equation*}
\end{proposition}
\begin{proof}
Since $\varphi_0$ is strictly positive, clearly $M_t^x$ is almost surely non-negative. Thus by the Feynman-Kac formula
\begin{equation*}
\E[(M_t^x)^+]=\E[M_t^x]=\varphi_0(x)\le \Norm{\varphi_0}{\infty}.
\end{equation*}
The martingale convergence theorem (see, e.g., \cite[Th. 2.10]{RY13}) implies that $\pro {M^x}$ has a final element
$M_\infty^x$ with $\E[|M_\infty^x|]<\infty$, and the optional stopping theorem (see, e.g., \cite[Th. 3.2]{RY13}) then
gives
\begin{equation*}
\varphi_0(x)=\E[M_0^x]=\E[M_\tau^x]=\E^x\big[e^{-\int_0^{\tau}(V(X_s)-\lambda_0)ds}\varphi_0(X_\tau)\big].
\end{equation*}
\end{proof}

\subsection{Symmetry properties}
Next we discuss some shape properties of ground states, specifically, symmetry and monotonicity, which will be
essential ingredients in the study of their local behaviour. First we show radial symmetry of the ground states 
for rotationally symmetric potential wells. This result can also been obtained by purely analytic methods, see 
\cite[Prop. 4.3]{AL3}.

\begin{theorem}\label{thm:rsym}
Let $\cK=\cB_a$ with a given $a>0$ and suppose that $H_{m,\alpha}$ has a ground state $\varphi_0$. Then $\varphi_0$ is
rotationally symmetric.
\end{theorem}
\begin{proof}
First observe that if another function $\widetilde{\varphi}_0$ existed satisfying \eqref{FKPW}, $\Vert \widetilde{\varphi}_0
\Vert_2=1$ and $\widetilde{\varphi}_0>0$, then by the uniqueness of the ground state we would have $\widetilde{\varphi}_0
\equiv \varphi_0$ almost surely.

Fix a rotation ${\sf R} \in SO(d)$ and consider $\widetilde{\varphi}_0(x)=\varphi_0({\sf R} x)$. Clearly, since
${\sf R}$ is an isometry, it is immediate that $\Vert \widetilde{\varphi}_0 \Vert_2=1$, $\widetilde{\varphi}_0>0$,
and $\widetilde{\varphi}_0(x)=\E[e^{-\int_0^t(V(X_s+{\sf R}x)-\lambda_0)ds}\varphi_0(X_t+ {\sf R}x)]$ by \eqref{FKPW}.
By rotational invariance of $\pro X$ we may furthermore write
\begin{equation*}
\widetilde{\varphi}_0(x)=\E\big[e^{-\int_0^t(V({\sf R} X_s+ {\sf R} x)-\lambda_0)ds}
\varphi_0({\sf R} X_t+{\sf R} x)\big] = \E\big[e^{-\int_0^t(V(X_s+x)-\lambda_0)ds}\widetilde{\varphi}_0(X_t+x)\big],
\end{equation*}
where we used the fact that also $V$ is rotationally invariant and $\cK=\cB_a$. Then by the observation above,
$\widetilde{\varphi}_0\equiv \varphi_0$ almost surely. Since ${\sf R} \in SO(d)$ is arbitrary, the claim follows.
\end{proof}

We can also prove a reduced symmetry of $\varphi_0$ for cases when $\cK$ is not spherically symmetric.
\begin{theorem}
\label{thm:asym}
Let $\cK$ be reflection symmetric with respect to a hyperplane $\cH$ such that $0 \in \cH$, and let ${\sf S}:
\R^d \to \R^d$, $x \mapsto {\sf S} x$, be such that ${\sf S} x$ is the reflection of $x$ with respect to $\cH$. Suppose
that $v$ is chosen such that $H_{m,\alpha}$ has a ground state $\varphi_0$. Then $\varphi_0({\sf S} x)=\varphi_0(x)$,
for all $x\in \R^d$.
\end{theorem}
\begin{proof}
We can argue similarly to Theorem \ref{thm:rsym}. Consider $\widetilde{\varphi}_0(x)=\varphi_0({\sf S} x)$. By the
isometry property of ${\sf S}$ we have again $\Vert \widetilde{\varphi}_0 \Vert_2=1$, $\widetilde{\varphi}_0>0$,
and $\widetilde{\varphi}_0(x)=\E[e^{-\int_0^t(V(X_s+{\sf S} x)-\lambda_0)ds}\varphi_0(X_t+ {\sf S} x)]$ by \eqref{FKPW}.
Since $\pro X$ is isotropic, we get
\begin{equation*}
\widetilde{\varphi}_0(x)=\E\big[e^{-\int_0^t(V({\sf S} X_s+ {\sf S} x)-\lambda_0)ds}\varphi_0({\sf S} X_t+
{\sf S} x)\big] = \E\big[e^{-\int_0^t(V(X_s+x)-\lambda_0)ds}\widetilde{\varphi}_0(X_t+x)\big],
\end{equation*}
where we used the fact that if $x \in \cK$, then also ${\sf S} x \in \cK$. Arguing as before, we obtain
$\varphi_0({\sf S} x)=\widetilde{\varphi}_0(x)=\varphi_0(x)$ for all $x\in \R^d$.
\end{proof}

\begin{remark}\label{rmk:rsym}
{\rm
We note that Theorems \ref{thm:rsym} and \ref{thm:asym} hold respectively for any rotationally or reflection
symmetric potential $V$ once a ground state exists and is unique. Moreover, they can be seen as converses to
\cite[Th. 7.1-7.2]{AL}, by using the expression
\begin{equation*}
V= - \frac{1}{\varphi_0}L_{m,\alpha}\varphi_0+\lambda_0,
\end{equation*}
provided $L_{m,\alpha}\varphi_0$ can be defined pointwise.
}
\end{remark}
We fix $\cK=\cB_a$ for some $a>0$ and assume that $H_{m,\alpha}$ has a ground state. Furthermore, we will make
extensive use of the following, for a proof see \cite{AL3}.
\begin{proposition}
\label{ass:mon}
There exists a non-increasing function $\rho_0:[0,\infty) \to \R$ such that $\varphi_0(x)=\rho_0(|x|)$ for every
$x \in \R^d$.
\end{proposition}
\noindent

\section{Local estimates}
\subsection{A prime example: Classical Laplacian and Brownian motion}
\label{primex}
First we present the case of the classical Schr\"odinger operator with a potential well, for which not only estimates can be
obtained but a full reconstruction of the ground state is possible by using the martingale $\pro M$ in (\ref{gsmartingale}).
Alternatively this can be done by an explicit solution of the Schr\"odinger eigenvalue
equation, which in this case is a textbook example, however, our point here is that while the eigenvalue problem cannot in
general be solved for non-local cases, the probabilistic approach is a useful alternative and this example shows best how
this can be done by using occupation times.
\begin{proposition}
Let
\begin{equation*}
H = -\frac{1}{2}\frac{d^2}{dx^2}  - v {\mathbf 1}_{\{|x| \leq a\}}
\end{equation*}
be given on $L^2(\R)$. Then the normalized ground state of $H$ is
\begin{equation*}
\varphi_0(x) = A_0 e^{-\sqrt{2|\lambda_0|}|x|} {\mathbf 1}_{\{|x| > a\}} +
B_0\cos(\sqrt{2(v-|\lambda_0|)} \,x) {\mathbf 1}_{\{|x| \leq a\}},
\end{equation*}
with
\begin{equation*}
A_0 = \sqrt{\frac{\sqrt{2|\lambda_0|}}{1+a\sqrt{2|\lambda_0|}}} e^{a\sqrt{2|\lambda_0|}} \cos(a\sqrt{2(v-|\lambda_0|)}),
 \quad B_0 = \sqrt{\frac{\sqrt{2|\lambda_0|}}{1+a\sqrt{2|\lambda_0|}}}.
\end{equation*}
\end{proposition}
\begin{proof}

Consider for any $b,c \in \R$ with $b<0<c$, the first hitting times
\begin{equation*}
T_{b} = \inf\{t>0: B_t = b\}, \quad T_{c} = \inf\{t>0: B_t = c\},  \quad \mbox{and} \quad T_{b,c} =
T_{b} \wedge T_c,
\end{equation*}
for Brownian motion $\pro B$ starting at zero, and recall the general formula by L\'evy \cite{L51}
\begin{equation*}
\ex^x[e^{iuT_{b,c}}] = \frac{e^{(1+i)x\sqrt u}}{e^{(1+i)c\sqrt u} + e^{(1+i)b\sqrt u }}
+ \frac{e^{-(1+i)x\sqrt u }}{e^{-(1+i)b\sqrt u} + e^{-(1+i)c\sqrt u}},
\end{equation*}
with $b < x < c$, and
\begin{equation}
\label{hittinglaplace}
\ex[e^{-uT_b}] = e^{-\sqrt{2u}|b|} \quad \mbox{and} \quad \ex[e^{-uT_{b,c}}] =
\frac{\cosh\left(\sqrt{2u}\,\frac{c+b}{2}\right)}{\cosh\left(\sqrt{2u}\, \frac{c-b}{2}\right)},
\quad u \geq 0.
\end{equation}
It is well-known that all these hitting times are almost surely finite stopping times with respect to the natural filtration. From (\ref{FKPW}) we have
\begin{equation*}
\varphi_0(x) =  \ex[e^{-|\lambda_0| t  + vU^x_t(a)} \varphi_0(B_t + x)],
\end{equation*}
where we denote
\begin{equation*}
U^x_t(a) = \int_0^t {\mathbf 1}_{\{|B_s+x| \leq a\}} ds =  \int_0^t {\mathbf 1}_{\{ -a-x \leq B_s\leq a-x\}} ds.
\end{equation*}
Then $U^x_{T_{-a-x,a-x}}(a) = T_{-a-x,a-x}$ whenever $|x| < a$, and is zero otherwise. Using Proposition \ref{prop:stop}
we obtain
\begin{eqnarray*}
\varphi_0(x)
&=&
\ex[e^{-|\lambda_0|T_{-a-x,a-x} + vU^x_{T_{-a-x,a-x}}(a)} \varphi_0(B_{T_{-a-x,a-x}}+x)].
\end{eqnarray*}
Now suppose $x > a$. By path continuity $T_{-a-x,a-x} = T_{a-x}$ and thus
\begin{equation*}
\varphi_0(x) = \ex[e^{-|\lambda_0|T_{a-x}} \varphi_0(B_{T_{a-x}}+x)] = \varphi_0(a) \ex[e^{-|\lambda_0|T_{a-x}}] =
\varphi_0(a) e^{-\sqrt{2|\lambda_0|}(x-a)}.
\end{equation*}
We obtain similarly for $x < -a$ that $T_{-a-x,a-x} = T_{-a-x}$ and
\begin{equation*}
\varphi_0(x) = \ex[e^{-|\lambda_0|T_{-a-x}} \varphi_0(B_{T_{-a-x}}+x)] = \varphi_0(-a) e^{-\sqrt{2|\lambda_0|}(-x-a)}
= \varphi_0(a) e^{\sqrt{2|\lambda_0|}(x+a)}
\end{equation*}
using $\varphi_0(-a)=\varphi_0(a)$. When $-a < x < a$, the two-barrier formula in (\ref{hittinglaplace})
gives
\begin{eqnarray*}
\varphi_0(x)
&=&
\ex[e^{(v-|\lambda_0|)T_{-a-x,a-x}} \varphi_0(B_{T_{-a-x,a-x}}+x)] \\
&=&
\ex[e^{(v-|\lambda_0|)T_{-a-x,a-x}} \varphi_0(B_{T_{-a-x,a-x}}+x) \textbf{1}_{\{T_{-a-x} < T_{a-x}\}}] \\
&& \qquad
+ \, \ex[e^{(v-|\lambda_0|)T_{-a-x,a-x}} \varphi_0(B_{T_{-a-x,a-x}}+x) \textbf{1}_{\{T_{-a-x} > T_{a-x}\}}] \\
&=&
\varphi_0(a) \frac{\cos(\sqrt{2(v-|\lambda_0|)}x)}{\cos(\sqrt{2(v-|\lambda_0|)}a)}.
\end{eqnarray*}
The constant $\varphi_0(a)$ can be determined by the normalization condition $\|\varphi_0\|_2=1$, which then yields the
claimed expression of the ground state.
\end{proof}

\begin{remark}
\label{CT}
{\rm
The argument can also be extended to higher dimensions. For instance, for $d \ge 3$, denote by $\cB_r(z)$ a ball of radius
$r$ centered in $z$, write $\cB_r = \cB_r(0)$, and
define the stopping times
\begin{equation*}
T_{r} = \inf\{t>0: X_t \in \bar{\cB}_r \} \quad \mbox{and} \quad \tau_{r} =
\inf\{t>0: X_t \notin  \cB_r \}.
\end{equation*}
Using the Ciesielski-Taylor formulae (see, e.g., \cite[eq. (2.15)]{CT62} and \cite[formula 2.0.1]{BS15})
\begin{equation*}
\ex^x[e^{-u\tau_{r}}] = \left(\frac{r}{|x|}\right)^{\frac {d-2}{2}}
\frac{I_{\frac{d-2}{2}}(|x|\sqrt{2u})}{I_{\frac{d-2}{2}}(r\sqrt{2u})}
\quad \mbox{and} \quad
\ex^x[e^{-uT_{r}}] = \left(\frac{r}{|x|}\right)^{\frac {d-2}{2}}
\frac{K_{\frac{d-2}{2}}(|x|\sqrt{2u})}{K_{\frac{d-2}{2}}(r\sqrt{2u})},
\end{equation*}
and the properties of the Bessel function $J_{(d-2)/2}$ and modified Bessel functions $I_{(d-2)/2}$ and $K_{(d-2)/2}$ in
standard notation (for properties of the Bessel functions, we refer to \cite{Wat}), by a similar argument as above for the
potential well $-v \textbf{1}_{\cB_a}$ we obtain
\begin{equation*}
\varphi_0(x) = A_0 \left(\frac{a}{|x|}\right)^{\frac {d-2} 2}
K_{\frac{d-2}2}\big(\sqrt{2|\lambda_0|}\, |x|\big) \textbf{1}_{\{|x| > a\}}
+ B_0 \left(\frac{a}{|x|}\right)^{\frac {d-2} 2}J_{\frac{d-2}2}\big(\sqrt{2(v-|\lambda_0|)}\,|x|\big)
{\mathbf 1}_{\{|x| \leq a\}},
\end{equation*}
where the constants $A_0, B_0$ can be determined from $L^2$-normalization as before. The details are left
to the reader.
}
\end{remark}

\subsection{Local behaviour of the ground state}
To come to our main point in this section, we need some scaling estimates on the L\'evy measure $\nu_{m,\alpha}$
of the exterior of a ball.
\begin{lemma}\label{lem:estLevy}
For every $R>0$ there exists a constant $C_{d,m,\alpha,R} > 1$ such that
\begin{equation*}
\int_{\cB_{C_{d,m,\alpha,R}R}^c}j_{m,\alpha}(|x-y|)dy\le \frac{1}{2}\int_{\cB_{R}^c}j_{m,\alpha}(|x-y|)dy.
\end{equation*}
Moreover, if $m=0$, then $C_{d,0,\alpha,R}$ does not depend on $R$.
\end{lemma}
\begin{proof}
Since $j_{m,\alpha}$ is non-increasing, for every $\theta>0$ the set $\{j_{m,\alpha}(|x|)\ge \theta\}$ is a ball
and then $\nu_{m,\alpha}(dx)$ is unimodal. As a consequence of Anderson's inequality \cite[Th. 1]{A55} we get
$\int_{\cB_R^c}j_{m,\alpha}(|x-y|)dy \ge \int_{\cB_R^c}j_{m, \alpha}(|y|)dy$, for every $R>0$ and $x \in \cB_R$.
Taking $R>0$, $x \in \cB_R$ and $k>2$, we obtain
\begin{align*}
\int_{\cB^c_{kR}}j_{m,\alpha}(|x-y|)dy&\le \int_{\cB^c_{(k-1)R}(x)}j_{m,\alpha}(|x-y|)dy\\&
= \int_{\cB^c_{(k-1)R}}j_{m,\alpha}((k-1)|y|)dy =(k-1)^d\int_{\cB^c_{R}}j_{m,\alpha}\left((k-1)|y|\right)dy.
\end{align*}
First consider $m=0$. We have
\begin{equation*}
\int_{\cB^c_{R}}j_{0,\alpha}((k-1)|y|)dy=\frac{1}{(k-1)^{d+\alpha}}\int_{\cB^c_{R}}j_{0,\alpha}(|y|)dy,
\end{equation*}
and thus
\begin{align*}
\int_{\cB^c_{kR}}j_{0,\alpha}(|x-y|)dy&
\le
\frac{1}{(k-1)^{\alpha}}\int_{\cB^c_{R}}j_{0,\alpha}\left(|y|\right)dy
\le
\frac{1}{(k-1)^{\alpha}}\int_{\cB^c_{R}}j_{0,\alpha}\left(|x-y|\right)dy.
\end{align*}
We can then set $C_{d,0,\alpha}=1+2^{1/\alpha}$ to complete the proof.

Next consider $m>0$. Using that
$j_{m,\alpha}(r)\sim C^{(2)}_{d,m,\alpha}r^{-\frac{d+\alpha+1}{2}}e^{-m^{1/\alpha}r}$ as $r \to \infty$,
we have
\begin{align*}
j_{m,\alpha}((k-1)|y|)
&\le
C_{d,\alpha,R}^{(3)} C^{(2)}_{d,m,\alpha}(k-1)^{-\frac{d+\alpha+1}{2}}|y|^{-\frac{d+\alpha+1}{2}}
\frac{e^{-m^{1/\alpha}(k-1)|y|}}{e^{-m^{1/\alpha}|y|}}e^{-m^{1/\alpha}|y|}\\
&\le
(C_{d,\alpha,R}^{(3)})^2 (k-1)^{-\frac{d+\alpha+1}{2}}e^{-m^{1/\alpha}kR}j_{m,\alpha}(|y|),
\end{align*}
with some $C^{(3)}_{d,\alpha,R}>1$, and hence
\begin{align*}
\int_{\cB^c_{kR}}j_{m,\alpha}(|y|) dy
&\le (C_{d,\alpha,R}^{(3)})^2(k-1)^{-\frac{d-\alpha-1}{2}}e^{-m^{1/\alpha}kR}\int_{\cB_R^c}j_{m,\alpha}(|y|)dy\\
&\le (C_{d,\alpha,R}^{(3)})^2(k-1)^{-\frac{d-\alpha-1}{2}}e^{-m^{1/\alpha}kR}\int_{\cB_R^c}j_{m,\alpha}(|x-y|)dy.
\end{align*}
Choosing $C_{d,m,\alpha,R}>2$ such that $(C_{d,\alpha,R}^{(3)})^2(C_{d,m,\alpha,R}-1)^{-\frac{d-\alpha-1}{2}}
e^{-m^{1/\alpha}C_{d,m,\alpha,R}R}\le \frac{1}{2}$ and using it instead of $k$, the claim follows.
\end{proof}
Combining the last estimate with the Ikeda-Watanabe formula, we obtain the following result.
\begin{lemma}
\label{lem:contrfun}
For every $R>0$ there exists a constant $C_{d,m,\alpha,R}>0$ such
that
\begin{equation*}
\E^x\big[g(\tau_{R}); R\le |X_{\tau_R}|\le C_{d,m,\alpha,R}R\big] \ge \frac{1}{2}\E^x[g(\tau_R)]
\end{equation*}
for every non-negative function $g$ and all $x \in \cB_R$.
\end{lemma}
\begin{proof}
First consider $g \in L^\infty(\R^d)$ and let $C_{d,m,\alpha,R}>0$ be defined as in Lemma \ref{lem:estLevy}.
By the Ikeda-Watanabe formula
\begin{equation*}
\E^x[g(\tau_R); |X_{\tau_R}|>C_{d,m,\alpha,R}R]=\int_0^{\infty}\int_{\cB_R}g(t)p_{\cB_R}(t,x,y)
\int_{\cB_{C_{d,m,\alpha,R}R}^c}j_{m,\alpha}(|y-z|)dzdydt.
\end{equation*}
Using Lemma \ref{lem:estLevy} we thus have
\begin{equation*}
\E^x[g(\tau_R); |X_{\tau_R}|>C_{d,m,\alpha,R}R]
\le
\frac{1}{2}\int_0^{\infty}\int_{\cB_R}g(t)p_{\cB_R}(t,x,y)\int_{\cB_{R}^c}j_{m,\alpha}(|y-z|)dzdydt
=\E^x[g(\tau_R)].
\end{equation*}
Next suppose that $g$ is unbounded and let $g_N(t)=g(t)\wedge N$ for $N \in \N$. Then $g_N \uparrow g$ pointwise,
moreover
\begin{equation*}
\E^x[g_N(\tau_R); R\le |X_{\tau_R}|\le C_{d,m,\alpha,R}R]\ge \frac{1}{2}\E^x[g_N(\tau_R)], \quad N \in \N.
\end{equation*}
As $N\to\infty$, by monotone convergence we then have
\begin{equation*}
\E^x[g(\tau_R); R\le |X_{\tau_R}|\le C_{d,m,\alpha,R}R]\ge \frac{1}{2}\E^x[g(\tau_R)].
\end{equation*}
\end{proof}

Now we can turn to local estimates of the ground state. Consider the spherical potential well supported in $\cK=\cB_a$
with some $a>0$.
\begin{theorem}
\label{mainth}
Let $\varphi_0$ be the ground state of $H_{m,\alpha}$ with $V=-v\mathbf{1}_{\cB_a}$ and denote
$\mathbf{a}=(a,0,\ldots,0)$. Then the estimates
\begin{equation*}
\varphi_0(x)\asymp \varphi_0(\mathbf{a})
\begin{cases}
\E^x[e^{(v-|\lambda_0|)\tau_a}] & \mbox{if \; $|x| \le a$}\\
\E^x[e^{-|\lambda_0|T_a}] & \mbox{if \; $|x| \ge a$}
\end{cases}
\end{equation*}
hold, where the comparability constant depends on $d,m,\alpha,a,v,\lambda_0$.
\end{theorem}
\begin{proof}
Note that $\varphi_0$ is rotationally symmetric by Theorem \ref{thm:rsym} and non-increasing by Proposition \ref{ass:mon}.
We first prove the bound inside and next outside the well.

\medskip
\noindent
\emph{Step 1:} First consider $|x| \le a$. Using Proposition \ref{prop:stop} with the almost surely finite stopping time $\tau_a$,
and that $X_{\tau_a}\in \cB_a^c$ and $\varphi_0(X_{\tau_a})\le \varphi_0(\mathbf{a})$, we have
\begin{equation}\label{eq:intpass1}
\varphi_0(x) = \E^x\big[e^{(v-|\lambda_0|)\tau_a}\varphi_0(X_{\tau_a})\big]
\le \varphi_0(\mathbf{a})\E^x\big[e^{(v-|\lambda_0|)\tau_a}\big].
\end{equation}
On the other hand, using that $|X_{\tau_a}|\le C^{(1)}_{d,m,\alpha,a}a$, where $C^{(1)}_{d,m,\alpha,a}$ is defined in Lemma
\ref{lem:contrfun}, we furthermore obtain
\begin{eqnarray*}
\varphi_0(x)
&\ge&
\E^x\big[e^{(v-|\lambda_0|)\tau_a}\varphi_0(X_{\tau_a}) ; a \le |X_{\tau_a}|\le C^{(1)}_{d,m,\alpha,a}a\big] \\
&\ge&
\varphi_0(C^{(1)}_{d,m,\alpha,a}\mathbf{a})\E^x\big[e^{(v-|\lambda_0|)\tau_a} ;  a \le |X_{\tau_a}|\le C^{(1)}_{d,m,\alpha,a}a\big].
\end{eqnarray*}
Recall that $C^{(1)}_{d,m,\alpha,a}>1$. Consider $T_{a}$ and $T_M=T_a \wedge M$ for any positive integer $M \in \N$. By Proposition
\ref{prop:stop} applied to the almost surely finite stopping time $T_M$, note that
\begin{equation*}
\varphi_0(C^{(1)}_{d,m,\alpha,a}\mathbf{a})=\E^{C^{(1)}_{d,m,\alpha,a}\mathbf{a}}[e^{-|\lambda_0|T_M}\varphi_0(X_{T_M})]\le \varphi_0(0)\E^{C^{(1)}_{d,m,\alpha,a}\mathbf{a}}[e^{-|\lambda_0|T_M}].
\end{equation*}
By dominated convergence, in the limit $M \to \infty$ we then get
\begin{equation*}
0<\varphi_0(C^{(1)}_{d,m,\alpha,a}\mathbf{a})\le  \varphi_0(0)\E^{C^{(1)}_{d,m,\alpha,a}\mathbf{a}}[e^{-|\lambda_0|T_a}],
\end{equation*}
implying $C^{(2)}_{d,m,\alpha,a}:=\bP^{C^{(1)}_{d,m,\alpha,a}\mathbf{a}}(T_a=\infty)<1$. In particular, there exists a constant
$C^{(3)}_{d,m,\alpha,a}>0$ such that $\bP^{C^{(1)}_{d,m,\alpha,a}\mathbf{a}}(T_a>C^{(3)}_{d,m,\alpha,a})<C^{(2)}_{d,m,\alpha,a}$.
Furthermore, by using Proposition \ref{prop:stop} again, we get
\begin{eqnarray*}
\varphi_0(C^{(1)}_{d,m,\alpha,a}\mathbf{a})
&=&
\E^{C^{(1)}_{d,m,\alpha,a}\mathbf{a}}[e^{-|\lambda_0|T_{M}}\varphi_0(X_{T_M})]\\
&\ge&
\E^{C^{(1)}_{d,m,\alpha,a}\mathbf{a}}[e^{-|\lambda_0|T_{a}}\varphi_0(X_{T_M})]
\ge \E^{C^{(1)}_{d,m,\alpha,a}\mathbf{a}}[e^{-|\lambda_0|T_{a}}\varphi_0(X_{T_M}) ; T_{a}\le C^{(3)}_{d,m,\alpha,a}].
\end{eqnarray*}
Since on the set $\{T_a \le  C^{(3)}_{d,m,\alpha,a}\}$ the random time $T_M$ is almost surely constant as $M \to \infty$,
in the limit
\begin{equation}\label{eq:intstep1}
\varphi_0(C^{(1)}_{d,m,\alpha,a}\mathbf{a})
\ge \E^{C^{(1)}_{d,m,\alpha,a}\mathbf{a}}[e^{-|\lambda_0|T_{a}}\varphi_0(X_{T_a}) ; T_{a}\le C^{(3)}_{d,m,\alpha,a}]
\ge (1-C^{(2)}_{d,m,\alpha,a})e^{-|\lambda_0|C^{(3)}_{d,m,\alpha,a}}\varphi_0(\mathbf{a})
\end{equation}
follows, where we also used Proposition \ref{ass:mon}. On the other hand, by Lemma \ref{lem:contrfun}, we have
\begin{equation}
\label{eq:intstep2}
\E^x\big[e^{(v-|\lambda_0|)\tau_a} ; a \le |X_{\tau_a}|\le C^{(1)}_{d,m,\alpha,a}a\big]
\ge \frac{1}{2}\E^x[e^{(v-|\lambda_0|)\tau_a}].
\end{equation}
Combining \eqref{eq:intstep1}-\eqref{eq:intstep2} with the above and choosing $C^{(4)}_{d,m,\alpha,a,|\lambda_0|}=
(1-C^{(2)}_{d,m,\alpha,a})e^{-|\lambda_0|C^{(3)}_{d,m,\alpha,a}}$ we obtain
\begin{equation*}
\varphi_0(x)\ge \frac{C^{(4)}_{d,m,\alpha,a,|\lambda_0|}}{2}\varphi_0(\mathbf{a})\E^x\big[e^{(v-|\lambda_0|)\tau_a}\big],
\end{equation*}
thus
\begin{equation*}
\varphi_0(x) \asymp \varphi_0(\mathbf{a})\E^x\big[e^{(v-|\lambda_0|)\tau_a}\big], \quad |x|\le a,
\end{equation*}
where the comparability constant depends on $d,m,\alpha,a,|\lambda_0|$.

\medskip
\noindent
\emph{Step 2:} Next consider $|x|>a$, and let $T_a$ and $T_M$ be defined as before. By Proposition \ref{prop:stop}
we have
\begin{align*}
\begin{split}
\varphi_0(x)=\E^x[e^{-|\lambda_0|T_{M}}\varphi_0(X_{T_M})] \ge \E^x[e^{-|\lambda_0|T_{a}}\varphi_0(X_{T_M})]
\ge \E^x[e^{-|\lambda_0|T_{a}}\varphi_0(X_{T_M}); T_a<\infty],
\end{split}
\end{align*}
due to $T_M \le T_a$. Taking the limit $M \to \infty$ and observing that $T_M$ is a definite constant if $T_a<\infty$, we
get
\begin{align}\label{eq:intpass2}
		\begin{split}
			\varphi_0(x)\ge \E^x[e^{-|\lambda_0|T_{a}}\varphi_0(X_{T_a}); T_a<\infty]\ge \varphi_0(\mathbf{a})\E^x[e^{-|\lambda_0|T_a};T_a<\infty]=\varphi_0(\mathbf{a})\E^x[e^{-|\lambda_0|T_a}].
		\end{split}
\end{align}
 On the other hand,
\begin{equation*}
\varphi_0(x)\le \varphi_0(0)\E^x[e^{-|\lambda_0|T_{M}}] \to \varphi_0(0)\E^x[e^{-|\lambda_0|T_{a}}],
\end{equation*}
as $M \to \infty$, by using dominated convergence. By Step 1, Theorem \ref{thm:CompMomGenFun} and \eqref{eq:upboundeig}
we find a constant $C^{(5)}_{d,m,\alpha,a,|\lambda_0|}$ such that
\begin{equation*}
\varphi_0(0) \le
C^{(5)}_{d,m,\alpha,a,|\lambda_0|}\varphi_0(\mathbf{a})\left(1+\frac{v-|\lambda_0|}{\lambda_a-v+|\lambda_0|}\right)
=:C^{(6)}_{d,m,\alpha,a,v,|\lambda_0|}\varphi_0(\mathbf{a}).
\end{equation*}
and thus
\begin{equation}\label{eq:intstep3}
\varphi_0(x)\le C^{(6)}_{d,m,\alpha,a,v,|\lambda_0|}\varphi_0(\mathbf{a})\E^x[e^{-|\lambda_0|T_{a}}].
\end{equation}
This leads to
\begin{equation*}
\varphi_0(x)\asymp \varphi_0(\mathbf{a})\E^x[e^{-|\lambda_0|T_{a}}], \quad |x|\ge a,
\end{equation*}
where the comparability constants depend on $d,m,\alpha,a,v,|\lambda_0|$.
\end{proof}

\begin{remark}
\label{rem:nolambda}
\hspace{100cm}
{\rm
\begin{enumerate}
\item 
In fact, along the way we also proved that
\begin{equation*}
C^{(1)}_{d,m,\alpha,a}\varphi_0(\mathbf{a})e^{-C^{(2)}_{d,m,\alpha,a}|\lambda_0|}\E^x[e^{(v-|\lambda_0|)\tau_a}]  \le
\varphi_0(x)\le C^{(3)}_{d,m,\alpha,a}\varphi_0(\mathbf{a})\E^x[e^{(v-|\lambda_0|)\tau_a}],
\end{equation*}
for every $|x|\le a$, with constants dependent only on $d,m,\alpha,a$ (and independent of $v$ and $\lambda_0$).
\item 
We point out that we have shown in particular that
\begin{equation*}
\E^x[e^{(v-|\lambda_0|)\tau_a}]\le \frac{2}{C^{(3)}_{d,m,\alpha,a,|\lambda_0|}} \frac{\varphi_0(x)}{\varphi_0(\mathbf{a})}
<\infty.
\end{equation*}
However, from \eqref{eq:momexitpass2} we know that $\E^x[e^{\lambda \tau_a}]$ is finite if and only if $\lambda<\lambda_a$.
Thus we have also shown that
\begin{equation}
\label{eq:upboundeig}
v-|\lambda_0|<\lambda_a.
\end{equation}
We note that to prove this only monotonicity of $\varphi_0$ outside the potential well is a required input, which has been 
proven in \cite{AL3} without using \eqref{eq:upboundeig} (which is, on the other hand, indispensable to obtain monotonicity 
inside the well). Thus this argument provides an alternative, purely probabilistic, proof of \cite[Lem. 4.5]{AL3}.	
\end{enumerate}
}
\end{remark}

Using the following estimate in conjunction with the estimates in Section 3 we can derive explicit local
estimates for the ground states of the massless and massive relativistic operators.
\begin{corollary}\label{cor:asymp}
With the same notations as in Theorem \ref{mainth} we have
\begin{equation*}
\varphi_0(x) \asymp
\; \varphi_0(\mathbf{a})\left\{
\begin{array}{lll}
1+\frac{v-|\lambda_0|}{\lambda_a-v+|\lambda_0|}\left(\frac{a-|x|}{a}\right)^{\alpha/2}
& \mbox{if \; $|x|\le a$} \vspace{0.2cm} \\
j_{m,\alpha}(|x|)
& \mbox{if \; $|x|\ge a$},
\end{array}\right.
\end{equation*}
where the comparability constant depends on $d,m,\alpha,a,v,|\lambda_0|$.
\end{corollary}

\begin{proof}
For $|x|\le a$ the result is immediate by a combination of Theorems \ref{mainth} and \ref{thm:CompMomGenFun},
using \eqref{eq:upboundeig}.
For $|x|\ge a$ we distinguish two cases. First, if $m=0$, by \cite[Cor. 4.1]{KL17} there exists
$R_{d,0,\alpha,a}$ such that
\begin{equation*}
\varphi_0(x)\ge C^{(1)}_{d,0,\alpha}|x|^{-d-\alpha}\ge C^{(2)}_{d,0,\alpha}j_{0,\alpha}(|x|),
\quad |x|\ge R_{d,0,\alpha,a},
\end{equation*}
where $C^{(1)}_{d,0,\alpha}$ is defined in the quoted result and
$C^{(2)}_{d,0,\alpha}=C^{(1)}_{d,0,\alpha}\frac{\pi^{d/2}\left|\Gamma\left(-\frac{\alpha}{2}\right)\right|}
{2^\alpha \Gamma\left(\frac{d+\alpha}{2}\right)}$.
Secondly, when $m>0$ we use \cite[Cor. 4.3(1)]{KL17} to find that there exists $R_{d,m,\alpha,a}$ such that
\begin{equation*}
\varphi_0(x)\ge C^{(1)}_{d,m,\alpha,a}|x|^{-\frac{d+\alpha+1}{2}}e^{-m^{1/\alpha}|x|}, \quad
|x| \ge R_{d,m,\alpha,a}.
\end{equation*}
Moreover, we know that $j_{m,\alpha}(x)\sim |x|^{-\frac{d+\alpha+1}{2}}e^{-m^{1/\alpha}|x|}$ as $|x| \to \infty$,
hence there exists a constant $C^{(2)}_{d,m,\alpha}$ such that
$\varphi_0(x)\ge C^{(2)}_{d,m,\alpha,a}j_{m,\alpha}(|x|)$ for $|x|\ge R_{d,m,\alpha,a}$.
Thus by \eqref{eq:intstep3}
\begin{equation*}
\E^x[e^{-|\lambda_0|T_a}]\ge C^{(3)}_{d,m,\alpha,a}j_{m,\alpha}(|x|), \quad |x|\ge R_{d,m,\alpha,a}.
\end{equation*}
Combining this with Corollary \ref{cor:goodlowbound} and Theorem \ref{thm:upboundLap}, we obtain
\begin{equation*}
\E^x[e^{-|\lambda_0|T_a}] \asymp j_{m,\alpha}(|x|), \quad |x|\ge a,
\end{equation*}
where the comparability constants depend on $d,\alpha,m,a,v,|\lambda_0|$.
\end{proof}

\begin{remark}
\label{rem:nolambda2}
{\rm
By Remark \ref{rem:nolambda} we have similarly
\begin{multline*}
C^{(1)}_{d,m,\alpha,a}\varphi_0(\mathbf{a})e^{-C^{(2)}_{d,m,\alpha,a}|\lambda_0|}
\left(1+\frac{v-|\lambda_0|}{\lambda_a-v+|\lambda_0|}\left(\frac{a-|x|}{a}\right)^{\alpha/2}\right)
\\
\le \varphi_0(x)\le C^{(3)}_{d,m,\alpha,a}\varphi_0(\mathbf{a})\left(1+\frac{v-|\lambda_0|}{\lambda_a-v+|\lambda_0|}
\left(\frac{a-|x|}{a}\right)^{\alpha/2}\right),
\end{multline*}
for $|x|\le a$ it holds and with constants which depend only on $d,m,\alpha,a$ (and not on $v$ and $\lambda_0$).
}
\end{remark}

The local estimates on $\varphi_0$ can further be improved to see the behaviour as $|x| \to a$.
\begin{proposition}
There exist $\varepsilon=\varepsilon_{d,m,\alpha,a,v},C_{d,m,\alpha,a,v}>0$ such that for every $x \in
\cB_{R+\varepsilon} \setminus\cB_{R-\varepsilon}$
\begin{equation*}
\left|\frac{\varphi(x)}{\varphi(\mathbf{a})}-1\right|\le C_{d,m,\alpha,a,v}\big| |x|-a \big|^{\alpha/2}
\end{equation*}
holds.
\end{proposition}
\begin{proof}
The estimate is clear once $x \in \partial \cB_a$. Consider first the case $x \in \cB_a$. By \eqref{eq:intpass1}
we have
\begin{equation*}
\frac{\varphi(x)}{\varphi(\mathbf{a})}-1\le \E^x[e^{(v-|\lambda_0|)\tau_a}-1]\le C_{d,m,\alpha,a,v}(a-|x|)^{\alpha/2},
\end{equation*}
where we used Theorem \ref{thm:CompMomGenFun}. Taking $x \in \cB_a^c$, we have by \eqref{eq:intpass2},
\begin{equation*}
1-\frac{\varphi(x)}{\varphi(\mathbf{a})}\le \E^x[1-e^{-|\lambda_0|T_a}].
\end{equation*}
Choosing $R^{(0)}_{d,m,\alpha,a,v}$ as in Proposition \ref{prop:halflower} and defining $\varepsilon=
(R^{(0)}_{d,m,\alpha,a,v}-a)\wedge a$ the result follows.
\end{proof}

By using the normalization condition $\Norm{\varphi_0}{2}=1$, we are able to provide a two-sided bound on
$\varphi_0(\mathbf{a})$.
\begin{proposition}
\label{ata}
Denote $\mathcal I = \int_1^{\infty}r^{d-1}j^2_{m,\alpha}\left(\frac{r}{a}\right)dr$ and by $B(x,y)$ the usual
Beta-function. Then
\begin{multline*}
\varphi_0(\mathbf{a}) \asymp
\left(a^d d\omega_d\left(\frac{1}{d}+2\frac{v-|\lambda_0|}{\lambda_a-v+|\lambda_0|}B\left(d,1+\frac{\alpha}{2}\right)
+\left(\frac{v-|\lambda_0|}{\lambda_a-v+|\lambda_0|}\right)^2B\left(d,1+\alpha\right)
+\mathcal I \right)\right)^{-\frac{1}{2}},
\end{multline*}
where the comparability constant is the same as in Corollary \ref{cor:asymp}.
\end{proposition}
\begin{proof}
We write $\kappa = \frac{v-|\lambda_0|}{\lambda_a-v+|\lambda_0|}$ for a shorthand.
Consider $|x| \le a$. By Corollary \ref{cor:asymp} we have
$$
\frac{1}{C_{d,m,\alpha,a,v,|\lambda_0|}}\varphi_0(\mathbf{a})\left(1+\kappa
\left(\frac{a-|x|}{a}\right)^{\frac{\alpha}{2}}\right)
\le
\varphi_0(x)
\le
C_{d,m,\alpha,a,v,|\lambda_0|}\varphi_0(\mathbf{a})
\left(1+\kappa\right)\left(\frac{a-|x|}{a}\right)^{\frac{\alpha}{2}},
$$
which gives
\begin{equation*}
\frac{1}{C_{d,m,\alpha,a,v,|\lambda_0|}}\varphi_0(x)
\le \varphi_0(\mathbf{a})\left(1+\kappa \left(\frac{a-|x|}{a}\right)^{\frac{\alpha}{2}}\right)
\le C_{d,m,\alpha,a,v,|\lambda_0|}\varphi_0(x).
\end{equation*}
Taking the square on both sides and integrating over $\cB_a$ we get
\begin{multline}
\label{eq:estphia1}
\frac{1}{(C_{d,m,\alpha,a,v,|\lambda_0|})^2}\int_{\cB_a}\varphi^2_0(x)dx \le \varphi_0^2(\mathbf{a})
\int_{\cB_a}\left(1+\kappa
\left(\frac{a-|x|}{a}\right)^{\frac{\alpha}{2}}\right)^2dx \\
\le (C_{d,m,\alpha,a,v,|\lambda_0|})^2\int_{\cB_a}\varphi^2_0(x) dx
\end{multline}
Consider next $|x|>a$. Proceeding similarly, we have
\begin{equation}
\label{eq:estphia2}
\frac{1}{(C_{d,m,\alpha,a,v,|\lambda_0|})^2}\int_{\cB^c_a}\varphi^2_0(x)dx \le
\varphi_0^2(\mathbf{a})\int_{\cB^c_a}j^2_{m,\alpha}(|x|)dx \\
\le (C_{d,m,\alpha,a,v,|\lambda_0|})^2\int_{\cB^c_a}\varphi^2_0(x) dx.
\end{equation}
Adding up \eqref{eq:estphia1}-\eqref{eq:estphia2} and using that $\Norm{\varphi_0}{2}=1$, we get
\begin{equation*}
\label{eq:estphia3}
\frac{1}{(C_{d,m,\alpha,a,v,|\lambda_0|})^2} \le
\varphi_0^2(\mathbf{a})\left(\int_{\cB_a}\left(1+\kappa\left(\frac{a-|x|}{a}\right)^{\frac{\alpha}{2}}\right)^2dx
+\int_{\cB^c_a}j^2_{m,\alpha}(|x|)dx\right) \\ \le (C_{d,m,\alpha,a,v,|\lambda_0|})^2.
\end{equation*}
Evaluating the integrals and taking the square root we obtain the desired result.
\end{proof}
As a direct consequence, we can rewrite Corollary \ref{cor:asymp} as follows.
\begin{corollary}\label{cor:asymp2}
	With the same notations as in Theorem \ref{mainth} we have
	\begin{equation*}
		\varphi_0(x) \asymp
		\; \left\{
		\begin{array}{lll}
			1+\frac{v-|\lambda_0|}{\lambda_a-v+|\lambda_0|}\left(\frac{a-|x|}{a}\right)^{\alpha/2}
			& \mbox{if \; $|x|\le a$} \vspace{0.2cm} \\
			j_{m,\alpha}(|x|)
			& \mbox{if \; $|x|\ge a$},
		\end{array}\right.
	\end{equation*}
	where the comparability constant depends on $d,m,\alpha,a,v,|\lambda_0|$ and is independent of $\varphi_0$.
\end{corollary}

\subsection{Lack of regularity of $\varphi_0$}
From a quick asymptotic analysis of the profile functions appearing in the estimates in Corollary \ref{cor:asymp}
the difference of the leading terms suggests that, while the regime change around the boundary of the potential
well is continuous, it cannot be smooth beyond a degree. To describe this quantitatively, we show next a lack of
regularity of the ground state arbitrarily close to the boundary. For a result on H\"older regularity of solutions
of related non-local Schr\"odinger equations see \cite{L16}.

\begin{lemma}\label{eq:propcont1}
Consider the operator $L_{m,\alpha}$ and the following two cases:
\begin{enumerate}
\item[(1)]
$\alpha \in (0,1)$ and $f \in C^{\alpha+\delta}_{\rm loc}(\R^d) \cap L^\infty(\R^d)$ for some $\delta
\in (0,1-\alpha)$
\vspace{0.1cm}
\item[(2)]
$\alpha \in [1,2)$ and $f \in C^{1,\alpha+\delta-1}_{\rm loc}(\R^d) \cap L^\infty(\R^d)$ for some $\delta \in
(0,2-\alpha)$.
\end{enumerate}
In either case above, the function $\R^d \ni x \mapsto L_{m,\alpha}f(x)$ is continuous.
\end{lemma}
\begin{proof}
Note that under the assumptions above, $L_{m,\alpha}f$ is well-defined pointwise via the integral representation \eqref{intrep}.
We show the statement for $m=0$ only, for $m>0$ the proof is similar by using the asymptotic behaviour of $j_{m,\alpha}(r)$
around zero and at infinity.

To prove (1), we use the integral representation \eqref{intrep} and claim that in this case
\begin{equation*}
L_{0,\alpha}f(x)=-C^{(1)}_{d,\alpha}\lim_{\varepsilon \downarrow 0}\left(\int_{\varepsilon<|x-y|<1}+\int_{|x-y|>1}\right)
\frac{f(y)-f(x)}{|x-y|^{d+\alpha}}dy = -C^{(1)}_{d,\alpha}\int_{\R^d}\frac{f(y)-f(x)}{|x-y|^{d+\alpha}}dy,
\end{equation*}
with the constant $C^{(1)}_{d,\alpha}$ entering the definition of the massless operator. Indeed, note that the second integral
in the split is independent of $\varepsilon$, while for the first integral we can use the H\"older inequality giving
\begin{equation*}
\int_{\varepsilon<|x-y|<1}\frac{|f(y)-f(x)|}{|x-y|^{d+\alpha}}dy\le C^{(2)}\int_{\varepsilon<|x-y|<1}
\frac{1}{|x-y|^{d-\delta}}\le dC^{(2)}\omega_d\int_{0}^1\frac{1}{\rho^{1-\delta}}d\rho=\frac{dC^{(2)}\omega_d}{\delta}.
\end{equation*}
The claimed right hand side follows then by dominated convergence. Next choosing $h \in \R^d$, $|h|<1$, we show that
$\lim_{h \to 0}L_{0,\alpha}f(x+h)=L_{0,\alpha}f(x)$. We write
\begin{align*}
L_{0,\alpha}f(x+h)&=-C^{(1)}_{d,\alpha}\int_{\R^d}\frac{f(y)-f(x+h)}{|x+h-y|^{d+\alpha}}dy
=-C^{(1)}_{d,\alpha}\left(\int_{\cB_3(x+h)}+\int_{\cB^c_3(x+h)}\right)\frac{f(y)-f(x+h)}{|x+h-y|^{d+\alpha}}dy.
\end{align*}
To estimate the first integral, note that $\cB_3(x+h)\subseteq \cB_4(x)$ for every $h \in \cB_1$. Let $C^{(3)}$ be
the H\"older constant associated with $\overline{\cB}_4(x)$ and observe that
\begin{align*}
\int_{\cB_3(x+h)}\frac{|f(y)-f(x+h)|}{|x+h-y|^{d+\alpha}}dy=\int_{\cB_3}\frac{|f(x+h+y)-f(x+h)|}{|y|^{d+\alpha}}dy
\le C^{(3)}\int_{\cB_3}\frac{dy}{|y|^{d-\delta}}=
\frac{3^\delta C^{(3)}d\omega_d}{\delta}.
\end{align*}
For the second integral, observe that if $y \in \cB_2(x)$, then $|x+h-y|\le |x-y|+|h|<3$ so that $y \in \cB_3(x+h)$ for
any $h \in \cB_1$. This means that $\cB^c_3(x+h)\subseteq \cB_2^c(x)$ for all $h$ and then
\begin{align*}
\int_{\cB^c_3(x+h)}\frac{|f(y)-f(x+h)|}{|x+h-y|^{d+\alpha}}dy&\le \int_{\cB^c_2(x)}\frac{|f(y)-f(x+h)|}{|x+h-y|^{d+\alpha}}dy\\
&\le 2\Norm{f}{\infty}\int_{\cB^c_2(x)}\frac{dy}{(|x-y|-|h|)^{d+\alpha}}\\
&\le 2\Norm{f}{\infty}d\omega_d \int_{2}^{\infty}\frac{\rho^{d-1}}{(\rho-1)^{d+\alpha}}d\rho<\infty.
\end{align*}
Thus again we can use dominated convergence to prove the claim.

Next consider (2). Fix $x \in \R^d$ and define the function
\begin{equation*}
\cB_1 \ni h  \mapsto D_hf(x):=f(x+h)-2f(x)+f(x-h).
\end{equation*}
By Lagrange's theorem there exist $\xi_\pm(h) \in [x,x \pm h]$, where $[x,y]$ denotes the segment with endpoints $x, y$,
such that
\begin{equation*}
f(x+h)-2f(x)+f(x-h)=\langle \nabla f(\xi_{+}(h))-\nabla f(\xi_{-}(h)), h\rangle
\end{equation*}
and thus
$|D_hf(x)|\le |\nabla f(\xi_{+}(h))-\nabla f(\xi_{-}(h))||h|$.
Since $\xi_\pm(h) \in [x,x \pm h]$, in particular $\xi_\pm(h) \in \cB_1(x)$, and we can use the H\"older property of
the gradient to conclude that
\begin{equation*}
|\nabla f(\xi_{+}(h))-\nabla f(\xi_{-}(h))|\le C^{(1)}(x)|\xi_{+}(h)-\xi_{-}(h)||h|^{\alpha+\delta-1}.
\end{equation*}
Moreover, $|\xi_{+}(h)-\xi_{-}(h)|\le 2$, and thus
$|D_hf(x)|\le 2C^{(1)}(x)|h|^{\alpha+\delta}$.
Using that $\int_{0}^{1}\frac{1}{\rho^{1-\delta}}d\rho=\frac{1}{\delta}$, by an application of \cite[Prop. 2.6, Rem. 2.4]{AL}
we then obtain
\begin{equation*}
L_{0,\alpha}f(x)=-\frac{C^{(2)}_{d,\alpha}}{2}\int_{\R^d}\frac{D_hf(x)}{|h|^{d+\alpha}}dh, \quad x \in \R^d.
\end{equation*}
Taking $k \in \cB_1$, we show that $\lim_{k \to 0}L_{0,\alpha}f(x+k)=L_{0,\alpha}f(x)$. Write
\begin{align*}
L_{0,\alpha}f(x+k) =-\frac{C^{(2)}_{d,\alpha}}{2}\int_{\cB_3}\frac{D_hf(x+k)}{|h|^{d+\alpha}}dh-
\frac{C^{(2)}_{d,\alpha}}{2}\int_{\cB_3^c}\frac{D_hf(x+k)}{|h|^{d+\alpha}}dh.
\end{align*}
In the first integral we have $x+k \pm h \in \cB_4(x)$ for every $k \in \cB_1$ and $h \in \cB_3$, hence
$|D_hf(x+k)|\le 8C^{(3)}(x)|h|^{\alpha+\delta}$,
similarly to in the previous case, where $C^{(3)}(x)$ is the H\"older constant of $\nabla f$ in $\overline{\cB}_4(x)$.
Thus we obtain
\begin{align*}
\int_{\cB_3}\frac{|D_hf(x+k)|}{|h|^{d+\alpha}}dh
\le \frac{8C^{(3)}d\omega 3^\delta}{\delta} \int_{\cB_3}\frac{dh}{|h|^{d-\delta}}.
\end{align*}
For the second integral, using that $f \in L^\infty(\R^d)$ we get
\begin{equation*}
\int_{\cB_3^c}\frac{|D_hf(x+k)|}{|h|^{d+\alpha}}dh\le 4\Norm{f}{\infty}\int_{\cB_3^c}\frac{dh}{|h|^{d+\alpha}}<\infty.
\end{equation*}
The proof is then completed by dominated convergence.
\end{proof}
\begin{theorem}
\label{irreg}
Let $\varphi_0$ be the ground state of $H_{m,\alpha}$. The following hold:
\begin{enumerate}
\item[(1)]
If $\alpha \in (0,1)$, then $\varphi_0 \not \in C^{\alpha+\delta}_{\rm loc}(\R^d)$ for every $\delta \in (0,1-\alpha)$.
\vspace{0.1cm}
\item[(2)]
If $\alpha \in [1,2)$, then $\varphi_0 \not \in C^{1,\alpha+\delta-1}_{\rm loc}(\R^d)$ for every $\delta \in (0,2-\alpha)$.
\end{enumerate}
\end{theorem}
\begin{proof}
We rewrite the eigenvalue equation like
\begin{equation}\label{eq:ScRev}
L_{m,\alpha}\varphi_0=(v\mathbf{1}_{\cB_a}+\lambda_0)\varphi_0.
\end{equation}
Suppose that $\alpha \in (0,1)$ and $\varphi_0 \in C^{\alpha+\delta}_{\rm loc}(\R^d)$ for some $\delta \in (0,1-\alpha)$.
Then by (1) of Lemma \ref{eq:propcont1} we have that the left-hand side of \eqref{eq:ScRev} is continuous. On the other hand,
take $\mathbf{e}_1=(1,0,\dots,0)$ and notice that
\begin{align*}
&\lim_{\varepsilon \downarrow 0}(v\mathbf{1}_{\cB_a}((a+\varepsilon)\mathbf{e}_1)
+\lambda_0)\varphi_0((a+\varepsilon)\mathbf{e}_1)=\lambda_0 \varphi_0(a\mathbf{e}_1)\\
&\lim_{\varepsilon \downarrow 0}(v\mathbf{1}_{\cB_a}((a-\varepsilon)\mathbf{e}_1)
+\lambda_0)\varphi_0((a-\varepsilon)\mathbf{e}_1)=(v+\lambda_0) \varphi_0(a\mathbf{e}_1),
\end{align*}
thus the right-hand side is continuous in $a\mathbf{e}_1$ if and only if $\varphi_0(a\mathbf{e}_1)=0$, which is in contradiction
with the fact that $\varphi_0$ is positive. In particular, the same argument holds for any point $x \in \partial \cB_a$, thus
the right-hand side of \eqref{eq:ScRev} has a jump discontinuity on $\partial \cB_a$, which is impossible since the left-hand
side is continuous. The same arguments hold for $\alpha \in [1,2)$ by using part (2) of Lemma \ref{eq:propcont1}.
\end{proof}
\begin{remark}
\label{bdryirreg}
\hspace{100cm}
{\rm
\begin{enumerate}
\item[(1)]
Instead of using $C_{\rm loc}^{\alpha+\delta}(\R^d)$ we also can prove part (1) of Lemma \ref{eq:propcont1} with $f \in
C^{\alpha+\delta}(\overline{\cB}_r(x))$ for some $x \in \R^d$, implying that $L_{m,\alpha}f$ is continuous in $x$. With this
localization argument we obtain for $\alpha \in (0,1)$ that $\varphi_0 \not \in C_{\rm loc}^{\alpha+\delta}(\cB_{a+\varepsilon}
\setminus \overline{\cB}_{a-\varepsilon})$, for all $\varepsilon \in (0,a)$ and $\delta \in (0,1-\alpha)$. In particular, this
implies that $\varphi_0$ cannot be $C^1$ on $\partial \cB_a$. The same arguments apply to part (2) of Lemma \ref{eq:propcont1}
and the case $\alpha \ge 1$, implying that $\varphi_0$ cannot be $C^2$ on $\partial \cB_a$. We note that for the classical
case the ground state is $C^1$ but fails to be $C^2$ at the boundary of the potential well.
\item[(2)]
It is reasonable to expect that $\varphi_0$ has at least a $C^{\alpha-\varepsilon}$-regularity, for all $\varepsilon >0$ small
enough, both inside and outside the potential well (away from the boundary). However, this needs different tools and we do not
pursue this point here.
\end{enumerate}
}
\end{remark}

\subsection{Moment estimates of the position in the ground state}
As an application of the local estimates of ground states we consider now the behaviour of the following functional.
Note that when the ground state is chosen to satisfy $\|\varphi_0\|_2=1$, the expression $\varphi_0^2(x)dx$ defines
a probability measure on $\R^d$. Let $p\geq 1$ and define
\begin{equation*}
\Lambda_p(\varphi_0) = \left(\int_{\R^d} |x|^p \varphi_0^2(x)dx\right)^{1/p},
\end{equation*}
which can then be interpreted as the $p$th moment of an $\R^d$-valued random variable under this probability
distribution. In the physics literature the ground state expectation for $p=2$ is called the size of the ground
state.

Let $m \ge 0$, $\alpha \in (0,2)$, and define
	\begin{equation*}
	p_*(m,\alpha):=\begin{cases}
	d+2\alpha & \mbox{if $m=0$}\\
	\infty &  \mbox{if $m>0$}.
	\end{cases}
	\end{equation*}

\begin{lemma}
The following cases occur:
\begin{enumerate}
\item
If $0<p<p_*(m,\alpha)$, then $\Lambda_p(\varphi_0)<\infty$.
\vspace{0.1cm}
\item
If $p \ge p_*(m,\alpha)$, then $\Lambda_p(\varphi_0)=\infty$.
\end{enumerate}
\end{lemma}
\begin{proof}
It is a direct consequence of Corollary \ref{cor:asymp}, using that $j_{0,\alpha}(r)=C_{d,\alpha}
r^{-d-\alpha}$, and $j_{m,\alpha}(r)\approx  r^{-(d+\alpha+1)/2}e^{-m^{1/\alpha}r}$ as $r \to \infty$ if $m>0$.
Indeed, while for $m>0$ it is immediate, for $m=0$ we have
$\rho^{d-1+p}j^2_{0,\alpha}(\rho)=C_{d,\alpha}\rho^{-(d+1+2\alpha-p)}$,
so that it is integrable at infinity if and only if $d+2\alpha>p$.
\end{proof}
\begin{proposition}
\label{moments1}
Let $0<p<p_*(m,\alpha)$. Then there exist constants $C^{(1)}_{d,m,\alpha,a,p},
C^{(2)}_{d,m,\alpha,a}>0$ such that
\begin{equation*}
\Lambda_p(\varphi_0) \ge C^{(1)}_{d,m,\alpha,a,p}\varphi^{2/p}_0(\mathbf{a})
\left(\frac{\lambda_a}{\lambda_a-v+|\lambda_0|}\right)^{2/p}e^{-\frac{2}{p}C^{(2)}_{d,m,\alpha,a}|\lambda_0|}.
\end{equation*}
\end{proposition}
\begin{proof}
By Remark \ref{rem:nolambda2} we get
\begin{align*}
\varphi_0^2(x)
&\ge \varphi^2_0(\mathbf{a})(C^{(3)}_{d,m,\alpha,a})^2\left(1+2\frac{v-|\lambda_0|}{\lambda_a-v+|\lambda_0|}
\left(\frac{a-|x|}{a}\right)^{\alpha/2}\right.\\
&\left.\qquad+\left(\frac{v-|\lambda_0|}{\lambda_a-v+|\lambda_0|}\right)^2
\left(\frac{a-|x|}{a}\right)^{\alpha}\right)e^{-2C^{(2)}_{d,m,\alpha,a}|\lambda_0|}\\
&\ge
\varphi^2_0(\mathbf{a})(C^{(3)}_{d,m,\alpha,a})^2\left(\frac{\lambda_a}{\lambda_a-v+|\lambda_0|}\right)^2
\left(\frac{a-|x|}{a}\right)^{\alpha}e^{-2C^{(2)}_{d,m,\alpha,a}|\lambda_0|},  \quad |x|\le a,
\end{align*}
where the last step follows by the fact that $\frac{a-|x|}{a} \le 1$. Hence
\begin{align*}
\int_{\R^d}|x|^p\varphi_0^2(x)dx&\ge \int_{\cB_a}|x|^p\varphi_0^2(x)dx\\
&\ge \varphi^2_0(\mathbf{a})(C^{(3)}_{d,m,\alpha,a})^2
\left(\frac{\lambda_a}{\lambda_a-v+|\lambda_0|}\right)^2e^{-2C^{(2)}_{d,m,\alpha,a}|\lambda_0|}
\int_{\cB_a}|x|^{p}\left(\frac{a-|x|}{a}\right)^{\alpha}dx.
\end{align*}
Setting
$(C^{(1)}_{d,m,\alpha,a,p})^p=(C^{(3)}_{d,m,\alpha,a})^2\int_{\cB_a}|x|^{p}\left(\frac{a-|x|}{a}\right)^{\alpha}dx$,
the result follows.
\end{proof}

\begin{proposition}
\label{moments2}
Let $0<p<p_*(m,\alpha)$ and $v>\lambda_a+\delta$ for some $\delta>0$. Then
there exists a constant $C_{d,m,\alpha,\delta,a,p}>0$ such that
\begin{equation*}
\Lambda_p(\varphi_0)
\le C_{d,m,\alpha,\delta,a,p}
\varphi^{2/p}_0(\mathbf{a})\left(\frac{\lambda_a}{\lambda_a-v+|\lambda_0|}\right)^{2/p}.
\end{equation*}
\end{proposition}
\begin{proof}
As in Theorem \ref{mainth}, observe that for $|x|\ge a$ we have by Proposition \ref{ass:mon}
\begin{equation}\label{eq:control4}
\varphi_0(x)\le \varphi_0(0)\E^x[e^{-|\lambda_0|T_a}].
\end{equation}
Moreover, by Remark \ref{rem:nolambda2},
\begin{equation}\label{eq:control5}
\varphi_0(0)\le
\frac{C^{(1)}_{d,m,\alpha,a}\lambda_a}{\lambda_a-v+|\lambda_0|}\varphi_0(\mathbf{a}).
\end{equation}
On the other hand, from $v-|\lambda_0|<\lambda_a$ we get $|\lambda_0|>v-\lambda_a>\delta$ and then
\begin{equation}\label{eq:control6}
\E^x[e^{-|\lambda_0|T_a}]\le \E^x[e^{-\delta T_a}]\le C^{(2)}_{d,m,\alpha,\delta,a}j_{m,\alpha}(|x|),
\quad |x|\ge a,
\end{equation}
where we used also Theorem \ref{thm:upboundLap}. Combining \eqref{eq:control5}-\eqref{eq:control6} with
\eqref{eq:control4}, we  obtain
\begin{equation}\label{eq:control7}
\varphi_0(x)\le \frac{C^{(3)}_{d,m,\alpha,\delta,a}\lambda_a}{\lambda_a-v+|\lambda_0|}\varphi_0(\mathbf{a})
j_{m,\alpha}(|x|), \quad |x|\ge a,
\end{equation}
where $C^{(3)}_{d,m,\alpha,\delta,a}=C^{(1)}_{d,m,\alpha,a}C^{(2)}_{d,m,\alpha,\delta,a}$.
For $|x|\le a$ we have directly by Remark \ref{rem:nolambda2}
\begin{align}\label{eq:control8}
\begin{split}
\varphi_0(x)
&\le C^{(1)}_{d,m,\alpha,a}\varphi_0(\mathbf{a})\left(1+\frac{v-|\lambda_0|}{\lambda_a-v+|\lambda_0|}
\left(\frac{a-|x|}{a}\right)^{\frac{\alpha}{2}}\right)
\le \frac{C^{(1)}_{d,m,\alpha,a} \lambda_a}{\lambda_a-v+|\lambda_0|} \varphi_0(\mathbf{a}),
\end{split}
\end{align}
where again we used that $\frac{a-|x|}{a}\le 1$. Hence by \eqref{eq:control7}-\eqref{eq:control8} we get
\begin{align*}
\int_{\R^d}|x|^p\varphi^2_0(x)dx&=	\int_{\cB_a}|x|^p\varphi^2_0(x)dx+	\int_{\cB_a^c}|x|^p\varphi^2_0(x)dx\\
&\le C^p_{d,m,\alpha,\delta,a,p}\left(\frac{\lambda_a}{\lambda_a-v+|\lambda_0|}\right)^2 \varphi^2_0(\mathbf{a}),
\end{align*}
where
\begin{equation*}
C^p_{d,m,\alpha,\delta,a,p}=\max\Big\{(C^{(1)}_{d,m,\alpha,a})^2\int_{\cB_a}|x|^pdx, \
(C^{(3)}_{d,m,\alpha,\delta,a})^2\int_{\cB_a^c}|x|^pj^2_{m,\alpha}(x)dx\Big\}.
\end{equation*}
\end{proof}

\begin{remark}
{\rm
As discussed in Section 2.2, a ground state exists for all $v>0$ when the process $\pro X$ is recurrent,
and it only exists for $v > v^*$ with a given $v^* = v^*(\alpha, m, a, d) > 0$ when the process is transient.
An interesting question is to analyze the blow-up rate of $\Lambda_p(\varphi_0)$ for some $p$ as $v \downarrow
v^*$. This would require a good control of the $v$-dependence of $\lambda_0$ and the comparability constants,
however, both appear to be rather involved. An expression of $\lambda_0=\lambda_0(v)$ may in principle be
expected to follow from the continuity condition $\varphi_0(\mathbf{a}-)=\varphi_0(\mathbf{a}+)$, however,
this seems to be difficult to obtain in any neat explicit form. In fact, even in the classical Schr\"odinger
eigenvalue problem this is a transcendental equation which can only numerically be solved, and the similar
blow-up problem also becomes untractable in terms of closed form expressions.
}
\end{remark}

\subsection{Extension to fully supported decaying potentials}
Our technique to derive local estimates on the ground state of a non-local Schr\"odinger operator with
a compactly supported potential can be extended to potentials supported everywhere on $\R^d$. This is
of interest since apart from decay rates as $|x|\to\infty$ (see \cite{KL17}), there is no information
on the behaviour of the ground state from small to mid range.

Consider a potential $V(x)=-v(|x|)$ with a continuous non-increasing function $v: \R^+ \to \R^+$ such that
$\lim_{r \to \infty}v(r)=0$. We assume that $H_{m,\alpha}$ has a ground state $\varphi_0$ with eigenvalue
$\lambda_0 < 0$. We already know from Remark \ref{rmk:rsym} that $\varphi_0$ is radially symmetric, thus
we can write $\varphi_0(x)=\varrho_0(|x|)$ with a suitable $\varrho_0:\R^+ \to \R^+$. Also in this case we
will suppose the following condition to hold.
\begin{assumption}
\label{ass:mon2}
The function $\varrho_0:[0,\infty) \to \R$ is non-increasing.
\end{assumption}

A first main result of this section is as follows.
\begin{theorem}\label{thm:main1noncomp}
Let $\varphi_0$ be the ground state of $H_{m,\alpha}$ with $V(x)=-v(|x|)$, $v:\R^+ \to \R^+$ non-increasing
and continuous. Let Assumption \ref{ass:mon2} hold and consider any $\gamma>0$ such that the level set
$\cK_\gamma=\{x \in \R^d: \ V(x)< -\gamma \} \neq \emptyset$. Then there exists a constant $C_
{d,m,\alpha,\gamma,|\lambda_0|}$ such that
\begin{equation*}
C_{d,m,\alpha,\gamma,|\lambda_0|}\varphi_0(x_\gamma)\E^x[e^{(\gamma-|\lambda_0|)\tau_{r_\gamma}}]
\le \varphi_0(x)\le \varphi_0(x_\gamma)\E^x[e^{(v(0)-|\lambda_0|)\tau_{r_\gamma}}], \quad x \in \cK_\gamma,
\end{equation*}
where $\tau_{r_\gamma}=\inf\{t>0: \ X_t \in \cK_\gamma^c\}$, $x_\gamma \in \partial \cK_\gamma$ is
arbitrary and $r_\gamma=|x_\gamma|$.
\end{theorem}
\begin{proof}
Take $x \in \cK_\gamma$ and notice that since $v$ is non-increasing and continuous, $\cK_\gamma$ is an open
ball centered at the origin, i.e., there exists $r_\gamma > 0$ such that $\cK_\gamma=\cB_{r_\gamma}$. Consider
the stopping time
\begin{equation*}
\tau_{r_\gamma}=\inf\{t>0: \ X_{t} \in \cB_{r_\gamma}^c\}.
\end{equation*}
Since $\varphi_0$ is radially symmetric, it is constant on $\partial \cB_{r_\gamma}$. Take any $x_\gamma \in
\partial \cB_{r_\gamma}$. By Proposition \ref{prop:stop} and Assumption \ref{ass:mon2} we have
\begin{equation*}
\varphi_0(x)=\E^x[e^{\int_0^{\tau_{r_\gamma}}v(|X_s|)ds-|\lambda_0|\tau_{r_\gamma}}\varphi_0(X_{\tau_{r_\gamma}})]
\le\varphi_0(x_\gamma)\E^x[e^{(v(0)-|\lambda_0|)\tau_{r_\gamma}}].
\end{equation*}
Consider $C^{(1)}_{d,m,\alpha,r_\gamma}>1$ defined in Lemma \ref{lem:contrfun} and observe that
\begin{align}
\begin{split}
\label{eq:contr0noncomp}
\varphi_0(x)&\ge \E^x[e^{\int_0^{\tau_{r_\gamma}}v(|X_s|)ds-|\lambda_0|\tau_{r_\gamma}}
\varphi_0(X_{\tau_{r_\gamma}});  r_\gamma \le X_{\tau_{r_\gamma}}\le C^{(1)}_{d,m,\alpha,r_\gamma}r_\gamma]\\
&\ge
\varphi_0(C^{(1)}_{d,m,\alpha,r_\gamma}x_\gamma)\E^x[e^{(\gamma-|\lambda_0|)\tau_{r_\gamma}};  r_\gamma \le
X_{\tau_{r_\gamma}}\le C^{(1)}_{d,m,\alpha,r_\gamma}r_\gamma].
\end{split}
\end{align}
Also, notice that by the definition of $C^{(1)}_{d,m,\alpha,r_\gamma}$,
\begin{equation}\label{eq:contr1noncomp}
\E^x[e^{(\gamma-|\lambda_0|)\tau_{r_\gamma}};  r_\gamma \le X_{\tau_{r_\gamma}}
\le C^{(1)}_{m,\alpha,r_\gamma}r_\gamma]\ge \frac{1}{2}\E^x[e^{(\gamma-|\lambda_0|)\tau_{r_\gamma}}].
\end{equation}
On the other hand, arguing as in Theorem \ref{mainth}, we have
\begin{equation*}
\varphi_0(C^{(1)}_{d,m,\alpha,r_\gamma}x_\gamma)\ge \E^{C^{(1)}_{d,m,\alpha,a}x_\gamma}[e^{-|\lambda_0|T_{r_\gamma}}
\varphi_0(X_{T_{r_\gamma}})],
\end{equation*}
where $T_{r_\gamma}=\inf\{t>0: \ X_{t} \in \cK_\gamma\}$ and we used the fact that $v(|x|)\ge 0$ for all
$x \in \R^d$. By Assumption \ref{ass:mon2} we have
\begin{equation}\label{eq:contr2noncomp}
\varphi_0(C^{(1)}_{m,\alpha,r_\gamma}x_\gamma)\ge \varphi_0(x_\gamma)\E^{C^{(2)}_{m,\alpha,a}x_\gamma}
[e^{-|\lambda_0|T_{r_\gamma}}]\ge C^{(2)}_{d,m,\alpha,r_\gamma,|\lambda_0|}\varphi_0(x_\gamma),
\end{equation}
where
\begin{equation*}
C^{(2)}_{d,m,\alpha,r_\gamma,|\lambda_0|}
:=C^{(3)}_{d,m,\alpha,r_\gamma,|\lambda_0|}j_{m,\alpha}(C^{(1)}_{d,m,\alpha,r_\gamma}r_\gamma),
\end{equation*}
$C^{(3)}_{d,m,\alpha,r_\gamma,|\lambda_0|}$ is defined in Corollary \ref{cor:goodlowbound} by choosing
$R_2> C_{d,m,\alpha,r_\gamma}^{(1)}r_\gamma$. Combining \eqref{eq:contr1noncomp}-\eqref{eq:contr2noncomp}
with \eqref{eq:contr0noncomp} the claim follows.
\end{proof}
\begin{remark}\label{rem:gamma0}
{\rm
We note that when $v(0)-|\lambda_0| \ge \lambda_{r_\gamma}$, the upper bound is trivial as
$\E^x[e^{(v(0)-|\lambda_0|)\tau_{r_\gamma}}] =\infty$. Also, if $|\lambda_0| \ge \gamma$, then the lower bound
is trivial since $\E^x[e^{(\gamma-|\lambda_0|)\tau_{r_\gamma}}] \le 1$ and $\varphi_0(x)\ge \varphi_0(x_\gamma)$
by Assumption \ref{ass:mon2}. Furthermore, by a similar argument as in Step 1 of Theorem \ref{mainth}, the implication
is that $\gamma-|\lambda_0|<\lambda_{r_\gamma}$ whenever $\cK_\gamma \neq \emptyset$. In particular, due to
$\lim_{\gamma \to v(0)}\lambda_{r_\gamma}=\infty$, there is a constant $\gamma_0>0$ such that $v(0)-|\lambda_0|
<\lambda_{r_\gamma}$ for every $\gamma \in (\gamma_0,v(0))$.
}
\end{remark}
Exploiting the asymptotic behaviour of the moment generating function involved as above for the spherical
potential well, we have the following result.
\begin{corollary}\label{cor:noncompsup}
Let $\varphi_0$ be the ground state of $H_{m,\alpha}$ with $V(x)=-v(|x|)$, $v:\R^+ \to \R^+$ non-increasing and
continuous. Let Assumption \ref{ass:mon2} hold, and consider any $\gamma>0$ such that the set $\cK_\gamma=\{x \in
\R^d: \ V(x)< -\gamma \} \neq \emptyset$, $|\lambda_0|<\gamma$ and $v(0)-|\lambda_0|<\lambda_{\cK_\gamma}$. Then
there exists a constant $C^{(1)}_{d,m,\alpha,\gamma,|\lambda_0|}$ such that
\begin{multline*}
C^{(1)}_{d,m,\alpha,\gamma,|\lambda_0|}\varphi_0(x_\gamma)\left(1+\frac{\gamma-|\lambda_0|}{\lambda_{\cK_\gamma}
-\gamma+|\lambda_0|}\left(\frac{r_\gamma-|x|}{r_\gamma}\right)^{\frac{\alpha}{2}}\right)\\ \le \varphi_0(x)
\le \varphi_0(x_\gamma)\left(1+\frac{v(0)-|\lambda_0|}{\lambda_{\cK_\gamma}-v(0)
+|\lambda_0|}\left(\frac{r_\gamma-|x|}{r_\gamma}\right)^{\frac{\alpha}{2}}\right),
\end{multline*}
for every $x \in \cK_\gamma$, where $\tau_{r_\gamma}=\inf\{t>0: \ X_t \in \cK_\gamma^c\}$, $x_\gamma \in \partial
\cK_\gamma$ is arbitrary, and $r_\gamma=|x_\gamma|$.
\end{corollary}
\begin{proof}
Starting from \eqref{thm:main1noncomp} and recalling that $\cK_\gamma=\cB_{r_\gamma}$, the upper bound follows from
the assumption that $v(0)-|\lambda_0|<\lambda_{r_\gamma}$ and Theorem \ref{thm:CompMomGenFun}. The lower bound
follows from Remark \ref{rem:gamma0} guaranteeing $\gamma-|\lambda_0|< \lambda_{r_\gamma}$, and furthermore by an
application of Theorem \ref{thm:CompMomGenFun}.
\end{proof}
\begin{theorem}\label{thm:mainoutnoncomp}
Let $\varphi_0$ be the ground state of $H_{m,\alpha}$ with $V(x)=-v(|x|)$, $v:\R^+ \to \R^+$ non-increasing and
continuous, and let Assumption \ref{ass:mon2} hold. Let $\gamma_1 \le |\lambda_0|$ and $\gamma_2 \in (\gamma_0,v(0))$,
where $\gamma_0$ is defined as in Remark \ref{rem:gamma0}, such that $\gamma_1 \le \gamma_2$. Define $\cK_{\gamma_i}
=\{x \in \R^d, V(x)<-\gamma_i\}$, $i=1,2$. Then
\begin{equation*}
\varphi_0(x_{\gamma_1})\E^x[e^{-|\lambda_0|T_{r_{\gamma_1}}}]\le \varphi_0(x)\le C_{d,m,\alpha,\gamma_2,|\lambda_0|}
\varphi_0(x_{\gamma_2}) \E^x[e^{(\gamma_1-|\lambda_0|)T_{r_{\gamma_1}}}], \quad x \in \cK_{\gamma_1},
\end{equation*}
where $x_{\gamma_i}\in \partial \cK_{\gamma_i}$ and $r_{\gamma_i}=|x_{\gamma_i}|$, $i=1,2$.
\end{theorem}
\begin{proof}
By a similar argument as in Theorem \ref{thm:main1noncomp}, there exist $r_{\gamma_i}$ such that $\cK_{\gamma_i}=
\cB_{r_{\gamma_i}}$, $i=1,2$. Moreover, $\cK_{\gamma_1}^c \subseteq \cK_{\gamma_2}^c$ since $v$ is non-increasing.
Let $x \in \cK_{\gamma_1}^c$ and observe that, as in Theorem \ref{mainth},
\begin{equation*}
\varphi_0(x_{\gamma_1})\E^x[e^{-|\lambda_0|T_{r_{\gamma_1}}}] \leq \varphi_0(x)
\le \varphi_0(0)\E^x[e^{(\gamma_1-|\lambda_0|)T_{r_{\gamma_1}}}],
\end{equation*}
where $x_{\gamma_1}\in \partial \cB_{r_{\gamma_1}}$. Using that $0 \in \cK_{\gamma_2}$, by Corollary \ref{cor:noncompsup}
we get
\begin{equation*}
\varphi_0(0)\le \varphi_0 (x_{\gamma_2})\left(1+\frac{v(0)-|\lambda_0|}{\lambda_{r_{\gamma_2}}-v(0)+|\lambda_0|}\right)=:
C_{d,m,\alpha,\gamma_2,|\lambda_0|}\varphi_0(x_{\gamma_2}), \quad x_{\gamma_2}\in \partial \cB_{r_{\gamma_2}}.
\end{equation*}
\end{proof}

Again, by using the asymptotics of the Laplace transform of the hitting times we get the following.
\begin{corollary}
Let $\varphi_0$ be the ground state of $H_{m,\alpha}$ with $V(x)=-v(|x|)$, $v:\R^+ \to \R^+$ non-increasing and
continuous, and let Assumption \ref{ass:mon2} hold. Choose $\gamma_1 \le |\lambda_0|$ and $\gamma_2 \in (\gamma_0,v(0))$,
where $\gamma_0$ is defined in Remark \ref{rem:gamma0}, such that $\gamma_1 \le \gamma_2$. Define $\cK_{\gamma_i}=
\{x \in \R^d, V(x)<-\gamma_i\}$, $i=1,2$. Then
\begin{equation*}
C^{(1)}_{d,m,\alpha,\gamma_1,|\lambda_0|}\varphi_0 (x_{\gamma_2})j_{m,\alpha}(|x|)\le \varphi_0(x)
\le C^{(2)}_{d,m,\alpha,\gamma_2,\gamma_1,|\lambda_0|}\varphi_0(x_{\gamma_2})j_{m,\alpha}(|x|),
\end{equation*}
where $x_{\gamma_i}\in \partial \cK_{\gamma_i}$ and $r_{\gamma_i}=|x_{\gamma_i}|$, $i=1,2$.
\end{corollary}
\begin{proof}
The upper bound follows directly by Theorems \ref{thm:mainoutnoncomp} and \ref{thm:upboundLap}. For the lower bound
first consider the potential well $\widetilde{V}=-\widetilde v\mathbf{1}_{\cK_{\gamma_1}}$, where $\widetilde v$ is
chosen to be large enough to guarantee the existence of a ground state $\widetilde{\varphi}_0$. Recall that
$\cK_{\gamma_1}$ is an open ball. By Corollary \ref{cor:asymp} we know that
\begin{equation*}
\frac{\widetilde{\varphi}_0(x)}{\widetilde{\varphi}_0(x_{\gamma_1})} \ge
C^{(3)}_{d,m,\alpha,\gamma_1,|\lambda_0|}j_{m,\alpha}(|x|), \quad x \in \cK_{\gamma_1}^c.
\end{equation*}
On the other hand, by Theorem \ref{mainth} we get
\begin{equation*}
\E^x[e^{-|\lambda_0|T_{r_{\gamma_1}}}]\ge
C^{(4)}_{d,m,\alpha,\gamma_1,|\lambda_0|}\frac{\widetilde{\varphi}_0(x)}{\widetilde{\varphi}_0(x_{\gamma_1})},
\quad x  \in \cK_{\gamma_1}^c.
\end{equation*}
Combining the previous estimates with the lower bound in Theorem \ref{thm:mainoutnoncomp}, the statement follows.
\end{proof}

\end{document}